\documentclass[a4paper,reqno,12pt]{amsart}

\usepackage[utf8]{inputenc}
\usepackage{amsmath, amssymb, amsthm, epsfig}
\usepackage{hyperref, latexsym}
\usepackage{url}
\usepackage[mathscr]{euscript}
\usepackage{color}
\usepackage{harpoon}
\usepackage{url}
\usepackage{mathabx}
\usepackage{tikz}
\usepackage{mathtools,fbox}
\DeclareMathAlphabet{\pazocal}{OMS}{zplm}{m}{n}

\usepackage{pgfplotstable}
\usepackage{pgfplots}

\usepackage{fullpage} 
\usepackage{setspace}
\usepackage{mathtools}
\mathtoolsset{showonlyrefs}

\usepackage{refcheck} 
\norefnames
\nocitenames

\def\today{\ifcase\month\or
  January\or February\or March\or April\or May\or June\or
  July\or August\or September\or October\or November\or December\fi
  \space\number\day, \number\year}

\newtheorem{theorem}{Theorem}
\newtheorem{conjecture}{Conjecture}
\newtheorem{lemma}[theorem]{Lemma}
\newtheorem{proposition}[theorem]{Proposition}
\newtheorem{corollary}[theorem]{Corollary}

\newtheorem{obs}{Remark}

\newtheorem*{prob}{Main Problem}

\newcommand{\A}{\mathcal{A}}

\newcommand{\D}{\mathcal{D}}

\newcommand{\F}{\mathcal{F}}
\newcommand{\G}{\mathcal{G}}

\newcommand{\M}{\mathcal{M}}

\renewcommand{\P}{\mathcal{P}}

\newcommand{\R}{\mathcal{R}}

\newcommand{\T}{\mathcal{T}}

\newcommand{\V}{\mathcal{V}}

\newcommand{\HH}{\mathbb{H}}

\newcommand{\z}{\mathbb{Z}}
\newcommand{\q}{\mathbb{Q}}
\renewcommand{\r}{\mathbb{R}}
\newcommand{\cp}{\mathbb{C}} 
\newcommand{\im}{{\rm Im}\,}
\newcommand{\re}{\Re}
\newcommand{\ft}{\widehat}

\newcommand{\bo}{\boldsymbol}

\newcommand{\wt}{\widetilde}

\newcommand{\la}{\lambda}
\newcommand{\ga}{\gamma}
\newcommand{\al}{\alpha}
\newcommand{\be}{\beta}
\newcommand{\ep}{\varepsilon}

\newcommand{\si}{\sigma}

\newcommand{\p}{\varphi}

\newcommand{\dens}{{\rm dens}}

\newcommand{\LP}{{\rm LP}}
\newcommand{\cc}{\mathfrak{c}}
\renewcommand{\b}{\mathfrak{b}}
\newcommand{\bd}{{\tfrac12}B_d}
\newcommand{\vol}{{\rm vol}}
\newcommand{\tu}[1]{{(#1)^{\rm trunc}}}
\newcommand{\td}[1]{{(#1)_{\rm trunc}}}
\renewcommand{\d}{\mathrm{d}}
\newcommand{\kiss}{\mathrm{kiss}}

\begin{document}


\title[]{Sphere Packings in Euclidean Space With Forbidden Distances}
\author[Gon\c{c}alves]{Felipe Gon\c{c}alves and Guilherme Vedana}
\date{\today}
\subjclass[2010]{}
\address{The University of Texas at Austin, 2515 Speedway, Austin, TX 78712, USA  \newline   {\&} IMPA - Instituto de Matemática Pura e Aplicada, Rio de Janeiro, 22460-320, Brazil.}
\email{felipe.ferreiragoncalves@austin.utexas.edu}
\address{IMPA - Instituto de Matemática Pura e Aplicada, Rio de Janeiro, 22460-320, Brazil.}
\email{guilherme.israel@impa.br}
\allowdisplaybreaks

\begin{abstract}
We study the sphere packing problem in Euclidean space where we impose additional constraints on the separations of the center points.
We prove that any sphere packing in dimension $48$, with spheres of radii $r$, such that \emph{no} two centers $x_1$ and $x_2$ satisfy $\sqrt{\tfrac{4}{3}} < \frac{1}{2r}|x_1-x_2| <\sqrt{\tfrac{5}{3}}$, has center density less or equal than $(3/2)^{24}$. Equality occurs for periodic packings if and only if the packing is given by a $48$-dimensional even unimodular extremal lattice. This shows that any of the lattices $P_{48p},P_{48q},P_{48m}$ and $P_{48n}$ are optimal for this constrained packing problem, and gives evidence towards the conjecture that extremal lattices are optimal unconstrained sphere packings in $48$ dimensions. We also provide results for packings up to dimension $d\leq 1200$, where we impose constraints on the distance between centers and on the minimal norm of the spectrum, showing that even unimodular extremal lattices are again uniquely optimal. Moreover, in the one-dimensional case, where it is not at all clear that periodic packings are among those with largest density, we nevertheless give a condition on the set of constraints that allows this to happen, and we develop an algorithm to find these periodic configurations by relating the problem to a question about dominos.
\end{abstract}

\maketitle

\tableofcontents

\section{Introduction}
The notorious \emph{Sphere Packing Problem} asks a simple question:  \emph{What is the best way of stacking higher-dimensional oranges\footnote{The authors' common favorite fruit.} in a higher-dimensional supermarket?} It is not too surprising that in dimension $3$, an optimal configuration arises when oranges (or cannonballs, which are not so juicy) are arranged in a hexagonal close packing, where laminated layers of spheres are assembled according to a suitable translation of the hexagonal lattice (fitting new spheres in the deep holes of the previous layer). What is remarkable is that this problem was proposed by Kepler around 1611 and was only solved in 1998 by Hales, in his famous large computer-assisted proof \cite{H05}. Recently, in 2016, the problem in dimensions $8$ and $24$ was solved by Viazovska et al. \cite{CKMRV17,Vi17}, introducing a remarkable new construction using quasi-modular forms to define certain smooth auxiliary functions $f_8$ and $f_{24}$ satisfying certain sign constraints in physical and frequency space to solve the problem.\footnote{Viazovska received the Fields medal in 2022 for her accomplishments.} To the best of our knowledge, there are only two other instances where Viazovska's technique was used to construct auxiliary functions that were indeed used in the solution of some kind of optimization problem: (1) in the solution of a $12$-dimensional uncertainty principle by the Cohn and Gon\c{c}alves \cite{CG19} (best constant and function); (2) in the proof of the universality of the $E_8$ and Leech lattices by Cohn, Kumar, Miller, Radchenko and Viazovska \cite{CKMRV21}. 


\subsection{Motivation} We systematically study a new type of constrained sphere packing problem, where we forbid certain short distances between centers of spheres. We are then able to solve exactly this problem when the forbidden set is the complement of a finite collection of square roots of even integers. The main questions driving/inspiring the problems we study and solve in this manuscript are the following:
\begin{itemize}
\item What other `natural' discrete geometry problems can be solved using Viazovska's modular forms construction technique? 
 
\item Can we embed the functions $f_8$ and $f_{24}$ found by Viazovska in a larger family of functions $\{f_{d}\}_{8|d}$ in such a way that these work as auxiliary functions that solve some kind of dual optimization problem of the question above?
\end{itemize}

\noindent{\bf Our answer}: \emph{Packings with forbidden distances and Theorems} \ref{thm:LP-bounds} \emph{and} \ref{thm:theguy}.

We now comment about our answer. First of all, the idea of creating a larger family of functions $\{f_{d}\}_{8|d}$ that contains Viazovska's functions was already explored in the interesting papers \cite{RW,FGH}. However, in these two papers, no geometrical optimization problem was solved, and the functions they generate are not related to ours. Also, it is worth pointing out that, as we shall see in the proof of Theorem \ref{thm:theguy}, there are several constraints our functions $\{H_d\}_{8|d}$ satisfy. The troublesome part is to show that certain sign conditions are met, and these do not follow from positivity of Fourier coefficients (as in \cite{FGH}, conjecturally), since they indeed change sign in our case. To overcome this, we came up with a numerical procedure, inspired by the one in \cite{CKMRV17}, and so our proof is unavoidably computer-assisted. Secondly, at first glance, one might say that imposing additional constraints in a sphere packing, such as forbidding certain distances, is esoteric or unnatural.  However, from the coding theory perceptive (sphere packings in $\mathbb{F}_q^m$), this question has been asked already, studied to some extent and has applications. To the best of our knowledge, we believe we were the first to consider this question in Euclidean space; nevertheless, in coding theory, codes with forbidden distances have been the object of study in many occasions. See, for instance, \cite{BCZ20,BZZPP,EFIN87,F}.  In Euclidean space, a cousin problem of the sphere packing problem is the chromatic number of $\r^d$, and in this venue mathematicians have considered already the version with forbidden distances; see, for instance, \cite{B,R,N23}. Thirdly, there are a bunch of unexpected features coming from our study that might drive further research, such as the following:

 $\heartsuit$ The functions $\{H_d\}_{8|d}$ we create in the proof of Theorem \ref{thm:theguy} have several curious properties that we have verified with a computer up to $d=1200$, but they lack proper mathematical explanation. Proving these properties propagate in every dimension $8|d$ would allow us to extend Theorem \ref{thm:theguy} to every dimension; 
 
  $\heartsuit$ The one-dimensional case we study in Section \ref{sec:1DD} is a very intriguing combinatorial/geometric problem that seems hard to analyze. The natural question here is to know when the best packing can be taken to be periodic, and we do provide a partial answer when the complement of the forbidden set is finite or has finitely many accumulation points (so to speak); 
  
  $\heartsuit$ Perhaps the most interesting contribution of our manuscript is Theorem \ref{thm:extremallattices1}, which we single out. It turns out that in dimension $48$, the constraints we need to impose are rather simple and nice, and we show that extremal lattices are optimal. Moreover, it is conjectured that extremal lattices are optimal unconstrained sphere packings in dimension $48$, so one can also see Theorem \ref{thm:extremallattices1} as further evidence to Conjecture \ref{conj:111}.
  
  $\heartsuit$  Since the submission of this paper, there have been further results on this topic that we would like to highlight. In \cite{BC}, Boyvalenkov  and Cherkashin find the largest kissing number with forbidden distances in dimension 48 - a configuration avoiding the set $(-1/3,-1/6)\cup (1/6,1/3)$. The related energy problem is investigated in \cite{BD}. Moreover, and most surprisingly, in \cite{BCD}, Boyvalenkov, Cherkashin and Dragnev find several types of distance avoiding optimal spherical codes in $\mathbb{S}^{15}, \mathbb{S}^{21}, \mathbb{S}^{22}$ and $\mathbb{S}^{23}$, via the linear programming method.


\subsection{Main results} As a prototype example of the kind of problem we will be concerned with, imagine that we are trying to place solid disks of diameter $1$ in $\r^2$, so to obtain the largest possible density. However, we require that either two disks kiss each other or their centers are far apart, say, with a distance not smaller than $\lambda>1$. As $\lambda$ slowly increases, we expect to see a transition between disks being allowed to `freely' move around and disks clumping together. Indeed, we show in Proposition \ref{prop:limitcase} that as $\la\to\infty$, the best arrangement is when three disks are placed on the vertices of an equilateral triangle of side length $1$ (kissing each other) and the circuncenters of these triangles are placed in a hexagonal lattice of side length approximately $\la$. In this paper we study a generalized version of this problem, where \emph{an arbitrary set of distances may be forbidden}.

We say that a sphere packing $P=X+r B_d$ ($B_d$ is the unit ball in $\r^d$ and $X$ the set of centers) \emph{avoids} a set $A\subset (1,\infty)$ if $|x_1-x_2|\notin 2 r A$ for all distinct $x_1,x_2\in X$ ($A$ is a set of \emph{forbidden distances}).  For instance, for the problem described in the previous paragraph, the set $A$ would be the interval $(1,\la)$. A periodic sphere packing is one where $X=\Lambda+Y$, $\Lambda$ is a lattice of minimal norm at least $2r$ and  $Y\subset \r^d/\Lambda$ is a nonempty finite set. A lattice packing is when $\#Y=1$. A lattice is even and unimodular if it has even squared norms and determinant $1$. Such lattice is said to be extremal if its minimal norm squared is equal to $2+2\lfloor d/24 \rfloor$ (see Section \ref{sec:main} for more information). We now state the first main result of this paper. 

\begin{theorem}\label{thm:extremallattices1}
Any even unimodular extremal lattice in $\r^{48}$ achieves maximal sphere packing density among all sphere packings that avoid the interval $\left(\sqrt{\tfrac43},\sqrt{\tfrac53}\right)$. Moreover, we have uniqueness among all periodic packings: if $P=\Lambda+Y+rB_d$ is some periodic sphere packing in $\r^{48}$ that avoids this interval and has maximal density, then $\frac{\sqrt{6}}{2r}(\Lambda+Y)$ is an even unimodular extremal lattice.
\end{theorem}

Our result shows that any of the lattices $P_{48p},P_{48q},P_{48m}$ and $P_{48n}$ are optimal for this constrained packing problem.
The first two lattices have a canonical construction as 2-neighbors of code lattices of extremal ternary codes (see \cite[p. 195]{CS} and \cite{N14} for the other two).

\begin{conjecture}\label{conj:111}
Any extremal lattice in dimension $48$ has maximal sphere packing density among all possible sphere packings.
\end{conjecture}

This conjecture is backed by the fact that no other better configuration is known. Perhaps an `easier' to prove conjecture is that extremal lattices in $48$ dimensions are the best \emph{lattice} packings. We believe Theorem \ref{thm:extremallattices1} could be used together with a computer-assisted method to reduce the amount of cases needed to be checked and show that $P_{48p}$ produces the best lattice packing; however, new ideas are needed here.  Below, we state a bold conjecture, which serves more as a research direction, as we have no numerical evidence towards it.

\begin{conjecture}
Let $L<\r^{48}$ be a lattice with  minimal norm $ \sqrt{6}$. If there is $x \in L$ with $\sqrt{8}<  |x| < \sqrt{10}$, then $L$ has covolume $> 1$.
\end{conjecture}

This conjecture in conjunction with Theorem \ref{thm:extremallattices1} implies that extremal lattices in $\r^{48}$ are the best lattice sphere packings. To see this, given any lattice $L$, normalize it so it has minimal norm $\sqrt{6}$. If there is no point $x\in L$ such that $\sqrt{8}<|x|<\sqrt{10}$, we then use Theorem \ref{thm:extremallattices1}; if such a point exists, we use the conjecture.

Theorem \ref{thm:extremallattices1} will follow from Theorems \ref{thm:LP-bounds} and \ref{thm:theguy}, where we develop a new linear programming method, similar to the Cohn and Elkies linear programming bound \cite[Theorem 3.1]{CE03}, and a generalization of Viazovska's modular function technique \cite{CKMRV17,Vi17} to find the desired `magic' function.  It turns out that if we allow ourselves to impose an extra condition on the spectrum of a given periodic configuration, one can prove a result similar to Theorem \ref{thm:extremallattices1} in every dimension $d$ multiple of $8$ not congruent to $16$ modulo $24$ up to $d=1200$.

Define the forbidden set
$$
A_d=(1,\sqrt{1+2/a_d})\cup (\sqrt{1+2/a_d},\sqrt{1+4/a_d}) \cup \ldots \cup (\sqrt{(l_d-2)/a_d},\sqrt{l_d/a_d}),
$$
where
\begin{equation}
\label{eq:def_a_d_l_d}
a_d=2+2\left\lfloor \frac{d}{24}\right\rfloor \ \ \text{ and }  \ \  l_d = a_d + 4\bigg(\bigg\lfloor \frac{d-4}{12} \bigg\rfloor-\bigg\lfloor \frac{d}{24} \bigg\rfloor\bigg).
\end{equation}
Our second main result is the following.

\begin{theorem}\label{thm:extremallattices2}
Let $8\leq d\leq 1200$, where $d$ is divisible by $8$ but $d\nequiv 16 \mod 24$. Let $P=\Lambda+Y+rB_d$ be some periodic sphere packing that avoids the set $A_d$ and such that  the minimal norm of $\Lambda^*$ is larger than $2r \sqrt{c_d}$, where $c_d$ is given in Table \ref{table:cd}. Then $$\dens(P)\leq \vol \left( B_d\right) \left( \frac{\sqrt{a_d}}{2}\right)^d.$$ Moreover, in case $\#Y=1$ then equality occurs if and only if $\frac{\sqrt{a_d}}{2r}\Lambda$ is an even unimodular extremal lattice.
\end{theorem}

Indeed,  Theorem \ref{thm:extremallattices1} can be seen as a particular case of Theorem \ref{thm:extremallattices2} since $c_{48}=0$ (hence we no longer need to assume that $P$ is periodic) and $A_{48}=(1,\sqrt{4/3})\cup (\sqrt{4/3},\sqrt{5/3})$, but the first interval can be removed because of a sign condition on the magic function of Theorem \ref{thm:theguy} . Theorem \ref{thm:extremallattices2} will follow from Theorem \ref{thm:LP-bounds} (new linear programming bounds), Theorem  \ref{thm:theguy} (magic contructions with  modular forms) and Theorem \ref{cor:theguy} (equivalent to Theorem \ref{thm:extremallattices2}). As in dimension $48$, one could reduce $A_d$ further for all $d$ by understanding the sign changes of the functions in Theorem \ref{thm:theguy}. However, there seems to be no particular interesting pattern, and the set $A_d$ would be rather complicated and given by a table other than a simple formula. Appealing to simplicity, we decided for the above form. The dimensions $d\equiv 16 \text{ mod } 24$ had to be excluded from our result since (for some unknown reason) the `magic' function we construct in these dimensions failS to satisfy some of the properties in Theorem \ref{thm:theguy}; for instance, their Fourier transform is nonpositive outside a neighborhood of the origin (although having positive mass). However, we believe it is possible to fix these issues if we impose more forbidden distances (see the more general Conjecture \ref{conj:evenlatt}). 

A classical result shows that extremal lattices may only exist up to dimension $d\leq 2\times 10^5$, but extending our results to such high dimensions seems out of reach with the present computational power on Earth, although we do believe they hold in all available dimensions (Conjecture \ref{conj:evenlatt}). Indeed, $d\leq 1200$ is an artifact of the computer-assisted part in this paper, but we believe it can be improved a little bit with cleverer/optimized algorithms.

Theorem \ref{thm:extremallattices2} puts forward a general framework and gives some kind of explanation to why one is only able to solve the sphere packing problem via linear programming in dimensions $8$ and $24$. We have now constructed a family of constrained problems, all amenable to linear programming methods and exact solutions via constructions with modular forms, which characterize extremal lattices as having optimal density among sphere packings avoiding certain distances. It is worth pointing out that the $E_8$ and Leech lattices are the only extremal latttices in dimensions $8$ and $24$, that $c_8=c_{24}=0$ and $A_8=A_{24}=\emptyset$. Hence, all the constraints we impose disappear in these dimensions, and we recover Viazovska's results. Curiously, the same set of (unscaled) distances  $\{\sqrt{m},\sqrt{m+1},...,\sqrt{n}\}$ appeared in a recent paper by Naslund on chromatic numbers of $\r^d$ \cite{N23}. Other remarks about our results are addressed in Section \ref{sec:main}.


\vspace{-2mm}
\section{Further main results}\label{sec:main}
We say that the set
$P=X+\bd$
is a sphere packing of $\r^d$ (associated to a set $X\subset \r^d$) if $|x-y|\in \{0\}\cup [1,\infty)$ for all $x,y\in X$, 
where $B_d:=\{y\in\r^d : |y|\leq 1\}$ is the unit ball and $|\cdot |$ is the Euclidean norm.  For a sphere packing $P$, we define its {density} by
\begin{align*}
    \dens(P):=\displaystyle\limsup_{t\rightarrow+\infty} \displaystyle \displaystyle\frac{\vol(P\cap tQ_d)}{\vol(tQ_d)},
\end{align*}
where $Q_d:=\left[-\frac{1}{2},\frac{1}{2}\right]^d$ is the unit cube. The setup is as follows. 
\smallskip

\noindent\fbox{%
  \begin{minipage}{\dimexpr\linewidth-2\fboxrule-2\fboxsep}%
\begin{prob}
Let $K\subset[1,+\infty)$ be a bounded subset such that $1\in K$. Consider
 the following family of sphere packings:
\begin{equation*}
    {\P}_d(K):=\left\{X+\bd : \forall x, y\in X\text{ we have } |x-y|\in \{0\}\cup K \cup  (\sup(K),\infty)\right\}.
\end{equation*}
The role of $K$ here is to prescribe the short distances between the centers of a sphere packing. We say a sphere packing $P$ is $K$-admissible if $P \in \P_d(K)$. What are the properties of $K$-admissible sphere packings $P$ that achieve maximal density? More precisely, if we let
 \begin{equation*}
     \Delta_d(K):=\sup_{P\in\pazocal{P}_d(K)} \dens(P),
 \end{equation*}
 we then want to study  $K$-admissible sphere packings $P$ such that $\dens(P)=\Delta_d(K)$. Alternatively, letting $A = (1,\sup(K)]\setminus K$, we then want to find a sphere packing of maximal density that avoids $A$; that is, no distance  between centers belongs to $A$.
\end{prob}
  \end{minipage}%
}

\vspace{5mm}
\noindent For now on, we will stick with the formulation using $K$ rather than $A$ (prescribing rather than forbidding), as it fits better our scheme of results and constructions. 

We now introduce some known facts about lattices.  A (full rank) lattice $\Lambda\subset \r^d$ is a discrete subgroup of $(\r^d,+)$ that contains $d$ linear independent vectors. We let 
$$
\ell(\Lambda):=\{|\la| : \la\in \Lambda\setminus\{0\}\}
$$
denote the lengths of $\Lambda$ and $\min \ell(\Lambda)$ denote its minimal norm. Given a lattice $\Lambda$ one can associate a sphere packing 
$$
P_\Lambda := \frac{1}{r}\Lambda + \bd \quad (\text{with } \ r=\min \ell(\Lambda))
$$
and show that $$\dens(\Lambda):=\dens(P_\Lambda)=\frac{\vol(\tfrac{r}{2}B_d)}{\vol(\r^d/\Lambda)}.$$ We say that a lattice $\Lambda$ is $K$-admissible if the packing $P_\Lambda$ above is $K$-admissible, that is, if
$$
\frac{\ell(\Lambda)}{\min \ell(\Lambda)} \subset K \cup (\sup(K),\infty).
$$
An even unimodular lattice $\Lambda$ is one such that $\vol(\r^d/\Lambda)=1$ and $\ell(\Lambda)\subset \{\sqrt{2n} : n\geq 1\}$ (such lattices are integral and self-dual). These lattices have been widely studied and classified in the literature. It is known that they can only exist in dimensions multiple of $8$ and, due to a classical theorem of Voronoi, that there are only finitely many of them in each dimension (modulo symmetries). It is known that (see \cite[p. 194, Cor. 21]{CS})
$$
(\min \ell(\Lambda))^2 \leq a_d := 2\bigg\lfloor \frac{d}{24} \bigg\rfloor +2.
$$
An even unimodular lattice attaining the above bound is called extremal . The $E_8$, $E_8^2, D_{16}^+$ and Leech lattices are the only even unimodular extremal lattices up to dimension $24$.
In dimensions $32$ and $40$, there are more than $10^7$ and $10^{51}$ of such lattices, respectively; however, in dimension $48$, there are (so far) only $4$ known lattices: $P_{48p},P_{48q},P_{48m}$ and $P_{48n}$ (see Nebe \cite{N14}). Moreover, it is known that extremal lattices cannot exist in sufficiently large dimensions \cite{MOS75} (as modular forms with several vanishing Fourier coefficients and very large weight necessarily have negative coefficients). The current best bound is due to Jenkins and Rouse \cite{JR11}, and it states that
$$
d_{\max} := {\sup}\{\text{rank}(\Lambda) : \Lambda \text{ is an even unimodular extremal lattice} \}\leq 163264.
$$
Indeed, one can show that for any $\be>0$, there exists $D$ such that there is no even unimodular lattice of rank $d > D$ and minimal squared norm larger than $a_{d}-\be$. For more information on extremal lattices, see \cite{SS99}.

We now state three other main results of this paper.  The first is an analogue of the Cohn and Elkies linear programming bound for $\Delta_d(K)$.
 
\begin{theorem}\label{thm:LP-bounds}
Let $K\subset[1,+\infty)$ be bounded and such that $1\in K$.
Define $$\Delta_d^{\LP}(K):=\vol \left( \bd\right)\inf\displaystyle\frac{F(0)}{\ft{F}(0)},$$ where the infimum is taken over all nonzero functions $F\in L^1(\r^d) \cap C(\r^d)$ such that :
\begin{align*}
F(x)  \leq 0 \ \text{ for }\ |x|\in K\cup(\sup(K),\infty) \quad \text{and} \quad
\ft{F}(x)  \geq 0  \text{ for all } x.
\end{align*}
Then
\begin{equation*}
    \Delta_d(K)\leq\Delta_d^{\LP}(K).
\end{equation*}
\end{theorem} 

\begin{theorem}\label{thm:theguy}
Let $8\leq d\leq 1200$, where $d$ is divisible by $8$ but $d\nequiv 16 \mod 24$. Define $l_d$ as in \eqref{eq:def_a_d_l_d} and 
$$
K_d = \frac{1}{\sqrt{a_d}}\{\sqrt{a_d},\sqrt{a_d+2},\sqrt{a_d+4},...,\sqrt{l_d}\}.
$$
Also let $c_d$ be given by Table \ref{table:cd} ($c_d=0$ if $d=8,24, 48$). Then there exists a nonzero radial function $H:\r^d\to \r$ of Schwartz class such that:
\begin{itemize}
\item $H(x)\leq 0$ if $|x|^2>l_d$;
\item $\ft H(x)\geq 0$ if $|x| > c_d$;
\item $H(x)=\ft H(x)=0$ if $|x|^2 \in \{a_d,a_d+2,...\}$;
\item $\{|x|^2 : H(x)=0 \text{ and } |x|^2 > l_d^-\}=\{l_d,l_d+2,...\}$;
\item $\{|x|^2 : \ft H(x)=0 \text{ and } |x|^2>c_d\}=\{a_d,a_d+2,...\}$
\end{itemize}
Moreover, if $d=48$ we additionally have that $\{|x|^2 : H(x) < 0\} \cap (0,10) =(6,8)$.
\end{theorem}

In the theorem above, $|x|^2>l_d^-$ means that $|x|^2>l_d-\epsilon_d$ for some small $\epsilon_d>0$. We note that one can indeed build functions $H$ for all dimensions congruent to $16$ modulo $24$ using the same techniques of Theorem \ref{thm:theguy}. However it turns out that $\ft H(x)\leq 0$ for $|x|>o(a_d)$ (numerically), although $H(x)\leq 0$ for $|x|^2>l_d$ and $H(0)=\ft H(0)>0$. One should also notice that the numbers $c_d$ seem to satisfy (for small $d$)
$$
c_d=a_d-2-O(1) \ \text{ if } d\equiv 8 \text{ mod } 24 \quad \text{and} \quad c_d= a_d-6-O(1) \ \text{ if }d\equiv 0  \text{ mod } 24.
$$
Also, in fact, $c_d$ is an approximation from the right of the last simple root of $\ft H(x)$. All these facts give a heuristic explanation why we only get results free from spectral conditions in dimensions $8,24$ and $48$ (hence a result for all sphere packings, periodic or not). It goes as follows: Experimentally, the $O(1)$ in $c_d$ is less than $1$ for small $d$ and $c_d$ increases with $d$ on each equivalence class modulo $24$, which means that if $d\geq 72$, then $a_d\geq 8$, and so $c_d\geq 1$. Thus $\ft H$ would never be nonnegative. For $d=8,24,48$, we have $c_d=0-O(1)<0$. Thus $\ft H\geq 0$ (see Figure \ref{fig:1}). For the remaining small dimensions not equal to $16$ modulo $24$, which are, $d=32$ and $d=56$, we have $c_{32}=2-O(1)$ and $c_{56}=4-O(1)$, which are positive, and so $\ft H$ is not nonnegative.

\begin{table}[h]
\begin{tabular}{cccccccccc}
\hline
$0$ & $\infty$ & $0$ & $1.5880$ & $\infty$ & $0$ & $3.5850$ & $\infty$ & $1.4710$ & $5.5790$ \\%
$\infty$ & $3.3760$ & $7.5720$ & $\infty$ & $5.1550$ & $9.5630$ & $\infty$ & $5.4650$ & $11.554$ & $\infty$ \\ %
$7.4160$ & $13.543$ & $\infty$ & $9.3400$ & $15.530$ & $\infty$ & $11.214$ & $17.514$ & $\infty$ & $11.462$ \\ %
$19.495$ & $\infty$ & $13.429$ & $21.469$ & $\infty$ & $15.384$ & $23.426$ & $\infty$ & $17.323$ & $23.537$ \\ %
$\infty$ & $19.234$ & $25.533$ & $\infty$ & $19.458$ & $27.530$ & $\infty$ & $21.433$ & $29.558$ & $\infty$ \\ %
$23.403$ & $31.519$ & $\infty$ & $25.358$ & $33.515$ & $\infty$ & $27.305$ & $37.454$ & $\infty$ & $29.244$ \\ %
$39.891$ & $\infty$ & $31.483$ & $42.000$ & $\infty$ & $39.277$ & $44$ & $\infty$ & $46$ & $46$ \\ %
$\infty$ & $48$ & $48$ & $\infty$ & $50$ & $50$ & $\infty$ & $52$ & $52$ & $\infty$ \\ %
$54$ & $54$ & $\infty$ & $56$ & $56$ & $\infty$ & $58$ & $58$ & $\infty$ & $60$ \\ %
$60$ & $\infty$ & $62$ & $62$ & $\infty$ & $64$ & $64$ & $\infty$ & $66$ & $66$ \\ %
$\infty$ & $68$ & $68$ & $\infty$ & $70$ & $70$ & $\infty$ & $72$ & $72$ & $\infty$ \\ %
$74$ & $74$ & $\infty$ & $76$ & $76$ & $\infty$ & $78$ & $78$ & $\infty$ & $80$ \\ %
$80$ & $\infty$ & $82$ & $82$ & $\infty$ & $84$ & $84$ & $\infty$ & $86$ & $86$ \\ %
$\infty$ & $88$ & $88$ & $\infty$ & $90$ & $90$ & $\infty$ & $92$ & $92$ & $\infty$ \\ %
$94$ & $94$ & $\infty$ & $96$ & $96$ & $\infty$ & $98$ & $98$ & $\infty$ & $100$ \\ \hline
\end{tabular}
\bigskip
\caption{Values of $c_d$ for $d=8,16,24,...,1200$. One should read it left to right top to bottom.  From dimension $d=536$ onwards, computational time was too high, and we simply took $c_d=a_d-2$, which can be verified much faster. For $d<536$, the numbers ${c_d}$ give a good rational approximation of the last sign change of the function $s\mapsto \ft H(\sqrt{s})$ from Theorem \ref{thm:theguy}. These numbers can also be found in the ancillary file \emph{cnumbers} on the arXiv submission of this paper (arXiv.org:2308.03925)}\label{table:cd}
\end{table}
\vspace{-0mm}

\begin{theorem}\label{cor:theguy}
Let $d$, $a_d$, $K_d$ and $c_d$ be as in Theorem \ref{thm:theguy}. Let $P=\Lambda+Y+\bd$ be some $K_{d}$-admissible periodic sphere packing such that $\min \ell(\Lambda^*)> \sqrt{c_d}$. Then $$\dens(P)\leq \vol \left( B_d\right) \left( \frac{\sqrt{a_d}}{2}\right)^d.$$  In case $\#Y=1$, equality above occurs if and only if $\sqrt{a_d}\Lambda$ is an even unimodular extremal lattice. We conclude that if $d\in\{8,24,48\}$, then
$$
\Delta_d(K_d) = \Delta_d^{\LP}(K_d) = \vol \left( B_d\right) \left( \frac{\sqrt{a_d}}{2}\right)^d.
$$ 
\end{theorem}

\begin{figure}[h]
\begin{tikzpicture}
\begin{axis}[
scale=1.5,
xtick={0,2,4,6,8,10},
ytick={1},
ymin=0,
ymax=10.5,
xmax=10,
xmin=0,
]
\addplot[
  solid,
]
table{H8};
\addplot[
  solid,
  blue,
]
table{H24};
\addplot[
  solid,
  red,
]
table{H48};
\end{axis}
\end{tikzpicture}
\caption{This is a plot of the functions $s\mapsto \ft H(\sqrt{s})e^{\pi s}$ for $d=8$ (black), $d=24$ (blue) and $d=48$ (red), normalized so $\ft H(0)=1$.}\label{fig:1}
\end{figure}
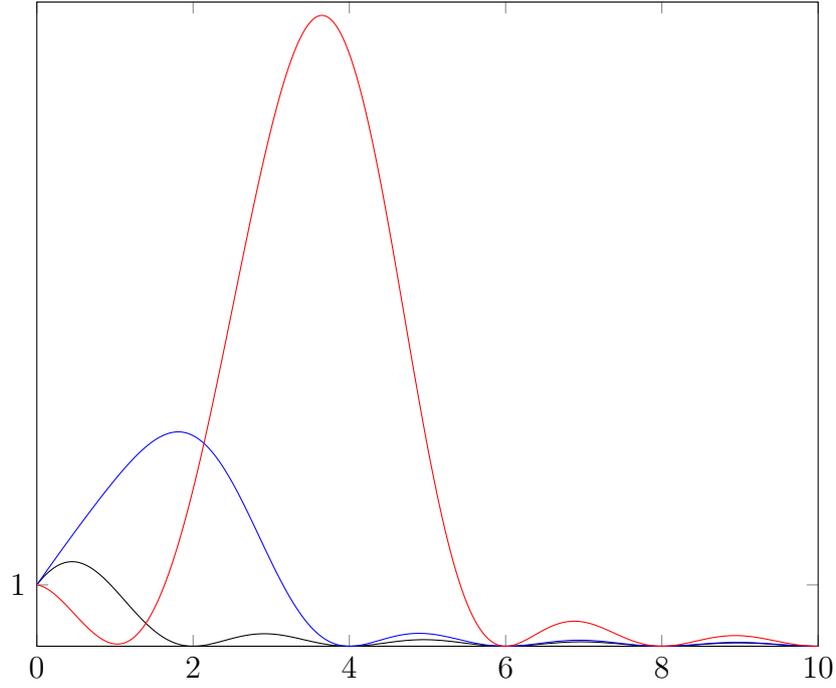

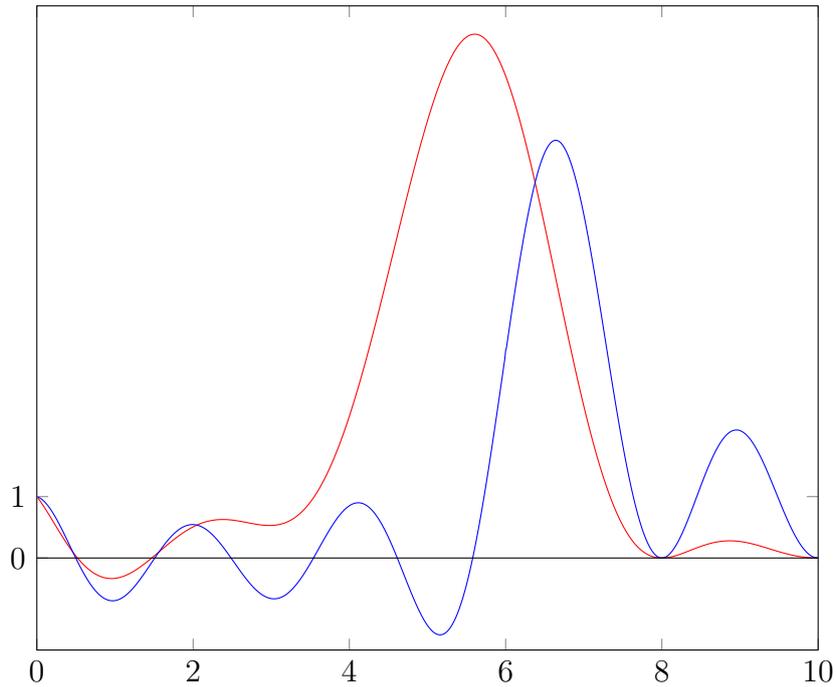
\begin{figure}[h]
\begin{tikzpicture}
\begin{axis}[
scale=1.5,
xtick={0,2,4,6,8,10},
ytick={0,1},
ymin=-1.5,
ymax=9,
xmax=10,
xmin=0,
]
\addplot[
  solid,
  red,
]
table{H72};
\addplot[
  solid,
  blue,
]
table{H80};
\addplot[mark=dot] coordinates {(0,0) (10,0)};
\end{axis}
\end{tikzpicture}
\caption{This is a plot of the functions $s\mapsto \ft H(\sqrt{s})e^{\pi s}$ for $d=72$ (red) and $d=80$ (blue), normalized so $\ft H(0)=1$. For $d=80$, we have multiplied the function by $(s+1)^2$ for aesthetic reasons.}\label{fig:2}
\end{figure}

Note that if $d\in \{8,24,48\}$, then $c_d=0$, and this shows that even unimodular extremal lattices maximize density for the among all $K_d$-admissible sphere packings (periodic or not). Since $K_8=K_{24}=\{1\}$, the packing problem in these dimensions is unconstrained, and the $E_8$ and Leech lattices are the only extremal lattices in these dimensions, this shows they have maximal density. The above theorem puts the results of \cite{CKMRV17,Vi17} in a larger family of packing problems that can be solved by linear programming methods and construction via modular forms. The fact that for $d=48$ we have $H(x) < 0$ for $6<|x|^2 < 8$, allows us to enlarge $K_{48}$ to $\frac1{\sqrt{6}}([\sqrt{6},\sqrt{8}]\cup {\sqrt{10}})$ and deduce Theorem \ref{thm:extremallattices1}. The same enlargement is possible in every dimension $d$; that is, we could fill $K_d$ between a couple of its points and prove a slightly stronger result. However, for simplicity, we left the statement as it is.

One could ask if it is possible to extend Theorems \ref{thm:theguy} and \ref{cor:theguy} to all dimensions $d\leq 163264$. For this, one would need to greatly optimize the numerical procedure we explain in the proof of Theorem \ref{thm:theguy}, and use specialized software and several days of running time to increase $1200$ to something of the order of $10000$. Rough experimental estimations show that the running time we have for the proof of Theorem \ref{thm:theguy} is roughly $O(1.1^{d/8})$-secs, so it seems the complexity of our algorithm is increasing exponentially. Even if one manages to reduce this ratio to (being generous) $1.001$,  reaching $d\approx 170000$ seems unreasonable. 

One could also ask if we could prove Theorem \ref{cor:theguy} with no assumption on the minimal norm on $\Lambda^*$. That might be possible, but we believe it to be impossible via the linear programming approach that we use with the same set $K_d$. The functions $H$ computed in Theorem \ref{thm:theguy} are in a way unique, and one could actually show they are by extending the interpolation formulas of \cite{CKMRV21} to all dimensions multiple of $8$. The issue is that the functions $\ft H(x)$ of Theorem \ref{thm:theguy} do have a simple zero very near $|x|^2=c_d$ and so have negative values in the region $0<|x|^2<c_d$.  However, it might be possible to remove the condition on $\Lambda^*$ by enlarging $l_d$ and finding the corresponding `magic' functions. We have tried this approach in small dimensions, replacing $l_d$ by $l_d+\delta$, for some small even $\delta>0$, although unsuccessfully. It might be the case that $\delta$ needs to be very large; however, that greatly complicates the modular form constructions. We leave this question for future work. Nevertheless, we expect this $\delta$ to exist because when $\delta=\infty$, the only $K_d$-admissible lattices  with minimal norm $\sqrt{a_d}$ are integral even lattices, and such lattices are less dense than extremal ones.

\begin{conjecture}\label{conj:evenlatt}
Let $\Lambda\subset \r^d$ be an even unimodular lattice with minimal norm $\sqrt{a}$, for some even integer $a$. Then for some even $l>a$, we have that $\Lambda$ has maximal density among any $\frac1{\sqrt{a}}\{\sqrt{a},\sqrt{a+2},...,\sqrt{l}\}$-admissible sphere packing; that is, 
$$
\Delta_d\left(\frac1{\sqrt{a}}\{\sqrt{a},\sqrt{a+2},...,\sqrt{l}\}\right)  = \dens(\Lambda).
$$
\end{conjecture}

If this conjecture is true in dimension $d$, then for $2\leq a \leq a_d$, one could define $L(a,d)=l$, where $l$ is the smallest such that Conjecture \ref{conj:evenlatt} is true. We already know that $L(2,8)=2, L(4,24)=4$, and we have shown that, $L(6,48)\leq 10$. We believe (and it is somewhat believed in the community) that $L(6,48)=6$ and that extremal lattices maximize density with no constraints in dimension $48$. It would be also interesting to find $L(2,16)$ and $L(2,24)$ and show that the lattices $E_8^2, D_{16}^+$ and all the $24$-dimensional Niemeier lattices with root are optimal.

Another curious question is: Is there a finite $K$ such that $\z^2$ is $K$-admissible and with maximal density?\footnote{This problem has a negative answer for dimensions $4$ or higher because by the Four Squares Theorem and scale invariance, the checkerboard lattice will always be $K$-admissible whenever $\z^d$ is, but it is a denser lattice.} If so, how small $\#K$ can be?

\subsection{One-dimensional sphere packings}\label{sec:1DD}
Unconstrained one-dimensional sphere packings are trivial to construct as unit intervals tile the line. However, finding optimal one-dimensional $K$-admissible sphere packings for an arbitrary given set $K$ seems to be a difficult question. Here, we are concerned with periodicity: \emph{When can we make sure that there exists some optimal one-dimensional packing which is periodic?} Unfortunately, greedy choice usually does not give an optimal construction. By greedy choice, we mean one starts with some configuration $\cup_{i=1}^N (a_i+I)$ for $a_1<a_2<...<a_{N}$ (with $I=[0,1]$) and then takes an interval $a_{N+1}+I$ with $a_{N+1}\geq 1+a_N$ as small as possible so that $\cup_{i=1}^{N+1} (a_i+I)$ still is $K$-admissible. For example, first consider the case where $K=\{1,\al\}$ for some $\al>2$. Then greedy choice gives the packing $P$ where one puts two unit intervals glued together, and then a gap of length $\al-1$, and the repeats this configuration periodically. One can see this is optimal by noting that, in any given $K$-admissible packing, the distance between the centers of any unit interval and the second one after it must be at least $1+\al$. This means that an interval of size $N(1+\al)$ contains at most $2N$ unit intervals, which shows that $\Delta_1(\{1,\al\}) \leq 2/(1+\al)$, and this is attained by $P$. However, this strategy does not produce the best packing in general. For instance, if $K=\{1,\al,\be \}$ with $2\al\geq \be>2\al-1>\al$, then greedy choice gives the packing $P_\be=I +(1+\be)\z+\{0,1\}$, which has density $2/(1+\be)$.  This is not optimal since $P_\al=I + \al \z $ has density $1/\al>2/(1+\be)$.  In Lemma \ref{thm:Kalbe}, we completely solve this problem for all choices of $\al$ and $\be$.

In Proposition \ref{prop:existence}, we show that when $K$ is a compact set, then optimal packings for $\Delta_d(K)$ always exist. However it is not guaranteed that they are periodic, as this is not even known in the unconstrained case. Nevertheless, we expect them to be periodic in the one-dimensional case. The following result proves this in the `almost' finite case.

\begin{theorem}\label{thm:1D}
Let $K\subset [1,\infty)$ be a compact set such that $1\in K$. Assume that $K$ has no accumulation points from the left and only finitely many accumulation points from the right. Moreover, let $K'$ be its set of accumulation points and assume that $(K'+K') \cap K'=\emptyset$.
Then there exists a $K$-admissible periodic sphere packing $P$ of $\r$ such that $\dens(P)=\Delta_1(K)$. 
\end{theorem}

In particular, optimal periodic $K$-admissible packings exist whenever $K$ is finite. However, optimal periodic packings will also exist in the (illustrative) case 
$$
K=\{1,\sqrt{2}, e,\pi\} +  \{10^{-n}\}_{n\geq 1}.
$$

\begin{conjecture}\label{conj:2}
Let $K\subset [1,\infty)$ be a compact set such that $1\in K$. Then there exists a $K$-admissible periodic sphere packing of $\r$ with maximal density.
\end{conjecture}

A compact set $K$ can be classified by its sequence of derived sets; that is, 
$$K,K', K'', K''',...,$$ where $S'$ is the set of points $p\in S$ such that $(p-\ep,p+\ep)\cap S \setminus \{ p\} \neq \emptyset$ for any $\ep>0$ (the accumulation points of $S$). Theorem \ref{thm:1D} solves the above conjecture for the case $K'=\emptyset$ (i.e., $K$ is finite) and deals with the case $K''=\emptyset$ (i.e., $K'$ is finite) under the condition that points only accumulate from the right and no accumulation point is a sum of two others.

We believe that Conjecture \ref{conj:2} could be very hard to prove, perhaps even false, as this is equivalent to (when $\sup(K)\in K$) a generalization of Theorem \ref{thm:finitedomino} (which is about linear domino tilings) for an infinite compact sets of symbols $\Sigma$ and domino pieces $\D\subset \Sigma^*\times \Sigma^*$.

\section{Generalities}

In this section, we establish some basic facts about sphere packing with forbidden distances. Throughout this section, $K\subset [1,\infty)$ will always be a bounded set such that $1\in K$, and the word sphere will be used to denote any  $x+\bd$ for some $x\in \r^d$.

\begin{proposition}
\label{prop:existence}
Assume $K$ is compact. Then there exists a packing $P\in\pazocal{P}_d(K)$ such that $\dens(P)=\Delta_d(K)$.
\end{proposition}
\begin{proof}
The proof is exactly the same as for unconstrained sphere packings \cite{G63}. Let  $(P_n)_{n\geq 1}$ be a maximizing sequence of sphere packings such that $\dens(P_n)$ is increasing and converges to $\Delta_d(K)$. By Proposition \ref{prop:periodic_packings_approximate_Delta}, we can assume that each $P_n$ is ${k_n} \z^d$ periodic for some integer $k_n>0$ such that $k_n\nearrow\infty$ and that $P_n$ maximizes the number of spheres one can put inside $k_nQ_d$. Using Hausdorff's topology for compact sets and a standard Cantor's diagonal argument, we can assume that $(P_n)_{n\geq 1}$ converges locally to some packing $P$, which is $K$-admissible since $K$ is compact. By maximality, the number of spheres of $P_n$ inside $k_mQ_d$ must not be much smaller than that of $P_m$ (the error must be bounded by the surface area of the boundary of $k_mQ_d$). We obtain
$$
\vol(P_n \cap k_m Q_d)/\vol(k_m Q_d) > \vol(P_m \cap k_m Q_d)/\vol(k_m Q_d) + O(1/k_m) =  \dens(P_m) + O(1/k_m)
$$
for all $n>m$, where $O(1/k_m)$ comes from the spheres that touch the boundary of $k_mQ_d$. Taking $n\to\infty$, we obtain $\vol(P \cap k_m Q_d)/\vol(k_m Q_d) \geq \dens(P_m) + O(1/k_m)$. Taking $m\to\infty$, we conclude that $\dens(P)\geq \Delta_d(K)$, which finishes the proof.
\end{proof}

\begin{proposition}
\label{prop:periodic_packings_approximate_Delta}
   Any $K$-admissible sphere packing can be approximated by a periodic one. In particular, if $N_t$ denotes the maximum number of spheres one can put inside $tQ_d$ such that the configuration is $K
$-admissible, then 
$$
\lim_{t\to\infty} \frac{N_t \vol(\bd)}{\vol(tQ_d)} = \Delta_d(K).
$$
\end{proposition}
\begin{proof}
Let $P=X+\bd$ be a $K$-admissible sphere packing.
Then $\wt P_t = X\cap tQ_d + (t+\sup(K))\z^d + \bd$  is $K$-admissible, periodic and
$$
\dens(\wt P_t)=\frac{\vol(P\cap tQ_d)}{\vol(tQ_d)}(1+O(1/t)).
$$
We obtain $\limsup_{t\to\infty} \dens(\wt P_t) = \dens(P)$. Let $\delta_t=\frac{N_t \vol(\bd)}{\vol(tQ_d)}$. The same periodization argument shows that $\limsup_{t\to\infty} \delta_t \leq \Delta_d(K)$. However, if $P=X+t\z^d$ is a $K$-admissible periodic sphere packing, by maximality, we must have that $\#(X \cap tQ_d) \leq N_t$; hence,
$$
\delta_t \geq \dens(P) + O(1/t),
$$
where $O(1/t)$ accounts for boundary intersections. We obtain $\liminf_{t\to\infty } \delta_t \geq \Delta_d(K)$. This finishes the lemma.
\end{proof}

\begin{lemma}\label{lem:enumerableK}
For any compact $K$, there is a countable set $\wt K$ such that $\Delta_d(K)=\Delta_d(\wt K)$.
\end{lemma}

\begin{proof}
Since $K$ is compact, there is a packing $P=X+\bd$ such that $\dens(P)=\Delta_d(K)$. Since $X$ is countable, we can write $X=\{x_1,x_2,...\}$. Define the set
\begin{equation*}
    K_0 =\{\al \in K : |x_i-x_j|=\al \text{ for some } i < j\}.
\end{equation*}
Define $\wt{K}:=K_0\cup\{1,\max(K)\}$. Then, $\wt{K}\subset K$ is a countable subset such that $\max( K)=\max(\wt K)$, and by construction, $P$ is $\wt K$-admissible. We have
$$
\Delta_d(\wt K)\leq \Delta_d(K) = \dens(P) \leq \Delta_d(\wt K).
$$
This concludes the proof.
\end{proof}

Define
\begin{equation*}
    n_d(K):=\max \{\#X : X\subset \r^d \text{ and } |x-y|\in K \text{ for all } x,y\in X \text{ with } x\neq y\} .
\end{equation*}
Since $K$ is bounded, it is clear that $X$ is finite and any maximal set $X$ (which always exist) can be placed inside a sphere of radii $\sup(K)$. For instance, if $K=\{1\}$, then $n_d(K)=d+1$, and this is realized by the $(d+1)$-simplex. Let $\kiss_d$ denote the kissing number of $\r^d$-that is, the largest number of equal size spheres that can touch a central sphere with no overlapping. Then it is easy to see that 
$$
n_d([1,2]) \geq 1+\kiss_d.
$$

\begin{conjecture}
For all $d$,  we have $n_d([1,2]) = 1+\kiss_d$.
\end{conjecture}

\noindent Trivially, this is attained for $d=1$. It seems to be the case for $d=2$ and unlikely to be false for $d=3$.  It turns out that the number $n_d(K)$ can be extracted from a constrained packing problem if one sets $K_\la = K \cup \{\la\}$ and sends $\la\to\infty$.

\begin{proposition}\label{prop:limitcase}
Let $K \subset [1,\infty)$ be bounded with $1\in K$ and let
$
K_\la = K \cup \{\la\}.
$
Then
$$
\lim_{\la\to\infty} \la^d \Delta_d(K_\lambda)=n_d(K)\Delta_d
$$
\end{proposition}

\begin{proof}
First, we claim that $\Delta_d(K_\lambda)\geq\frac{n_d(K)\cdot \Delta_d}{\left(\beta+\lambda\right)^d}$, where $\be=\sup{K}$. In order to do that, we will construct a packing $P_\la$ that is $K_\la$-admissible and show that $\dens(P_\lambda)\geq\frac{n_d(K)\cdot \Delta_d}{\left(\beta+\lambda\right)^d}$. The packing $P_\lambda$ will not necessarily have maximal density; however it will have a nice structure which makes it easy to estimate its density.  Let $Y\subset \be B_d$ be a maximal cluster of $K$-admissible points attaining $\#{Y}=n_d(K)$.  Let $\wt{P_\lambda}=X_\lambda+\frac{\be+\la}{2} B_d$ be an unconstrained periodic sphere packing (with spheres of diameter $\be+\la$) such that $\dens(\wt{P_\lambda})>\Delta_d-\ep$. Define the packing 
$$
P_\lambda = X_\la + Y+\bd 
$$
We claim $P_\la$ is $K_\la$ admissible. To see this note that if $x_\la+y$ and $x_\la'+y'$ are two points in $X_\la+Y$, then their distance is $\geq \la$ if $x_\la\neq x_\la'$. If $x_\la=x_\la'$, then their distance is $|y-y'|\in K$. We obtain
\begin{align*}
  \dens(P_\la) &  = \lim_{t\to\infty} \frac{\#((X_\la+Y)\cap tQ_d) \vol(\bd)}{\vol(tQ_d)} \\ & =  \lim_{t\to\infty} \frac{\#(X_\la\cap (t-2\be -1)Q_d) n_d(K)\vol(\frac{\be+\la}{2} B_d)}{(\be+\la)^d\vol(tQ_d)} \\
  & = \frac{n_d(K)\dens(\wt{P_\lambda})}{(\be+\la)^d} > \frac{n_d(K)(\Delta_d-\ep)}{(\be+\la)^d}.
\end{align*}
Since both $\wt{P_\la}$ and $P_\la$ are periodic, equality between limits above is justified. Letting $\ep\to 0$ proves our claim.

Now we claim that $\Delta_d(K_\lambda)\leq\frac{n_d(K)\cdot\Delta_d}{\lambda^d}$.  Let $Y+\bd$ be a periodic $K_\la$-admissible sphere packing such that $\dens(Y_\la)>\Delta_d(K_\la)-\ep$. We can assume that $\la>2\be$.  Define an equivalence relation in $Y$ by saying that $y_1 \sim y_2$ if $|y_1-y_2|\leq \be$. This is an equivalence relation since if $y_1\sim y_2$ and $y_2\sim y_3$ but $|y_1-y_3|>\be$, then $|y_1-y_3|\geq \la$, but triangle inequality shows that $|y_1-y_3|\leq 2\be < \la$, a contradiction. Let $\wt Y=\{[y_1],[y_2],...\}$ be these equivalence classes, where the $y_j$'s are representatives of each class. Observe that $|y_i-y_j|\geq \la$ if $i<j$, that the set $[y_j]$ has only distances in $K$ and is contained in $y_j + {\be}B_d$. Thus, $\#[y_j]\leq n_d(K)$. We obtain
\begin{align*}
 \Delta_d(K_\lambda) -\ep & <  \lim_{t\to\infty} \frac{\#(Y\cap tQ_d) \vol(\bd)}{\vol(tQ_d)} \\ & \leq \limsup_{t\to\infty} \frac{ \vol(\bd) \sum_{ [y_j]\cap tQ_d \neq \emptyset } \#[y_j]}{\vol(tQ_d)} \\ & \leq \frac{n_d(K)}{\la^d} \limsup_{t\to\infty} \frac{\#(\{y_1,y_2,...\}\cap (t+2\be)Q_d) \vol(\tfrac{\la}{2}B_d)}{\vol(tQ_d)} \\
 & \leq \frac{n_d(K)}{\la^d} \Delta_d.
\end{align*}
Letting $\ep\to 0$ finishes the proof.
\end{proof}

Observe that we have actually proven the stronger result
$$
\frac{n_d(K)\cdot \Delta_d}{\left(\beta+\lambda\right)^d} \leq \Delta_d(K_\lambda) \leq \frac{n_d(K)\cdot \Delta_d}{\la^d}.
$$
In particular, in conjunction with Theorem \ref{thm:LP-bounds}, we have the new linear program to produce upper bounds for kissing numbers.

\begin{corollary}
Let $\la\geq 4$. Let $F:\r^d\to\r$ be a continuous $L^1$-function such that $\ft F\geq 0$ and $F(x)\leq 0$ if $1<|x|<2$ or if $|x|>\la$. Then 
$$
1+\kiss_d \leq \frac{(2+\la)^d}{\Delta_d} \frac{F(0)}{\ft F(0)} \vol(\bd).
$$
\end{corollary}

The structure of the best configuration for $d=2$, $K=[1,2]\cup \{\la\}$ and large $\la$ will look like Figure \ref{fig:la4}.

\begin{figure}[h]
\includegraphics[scale=.6]{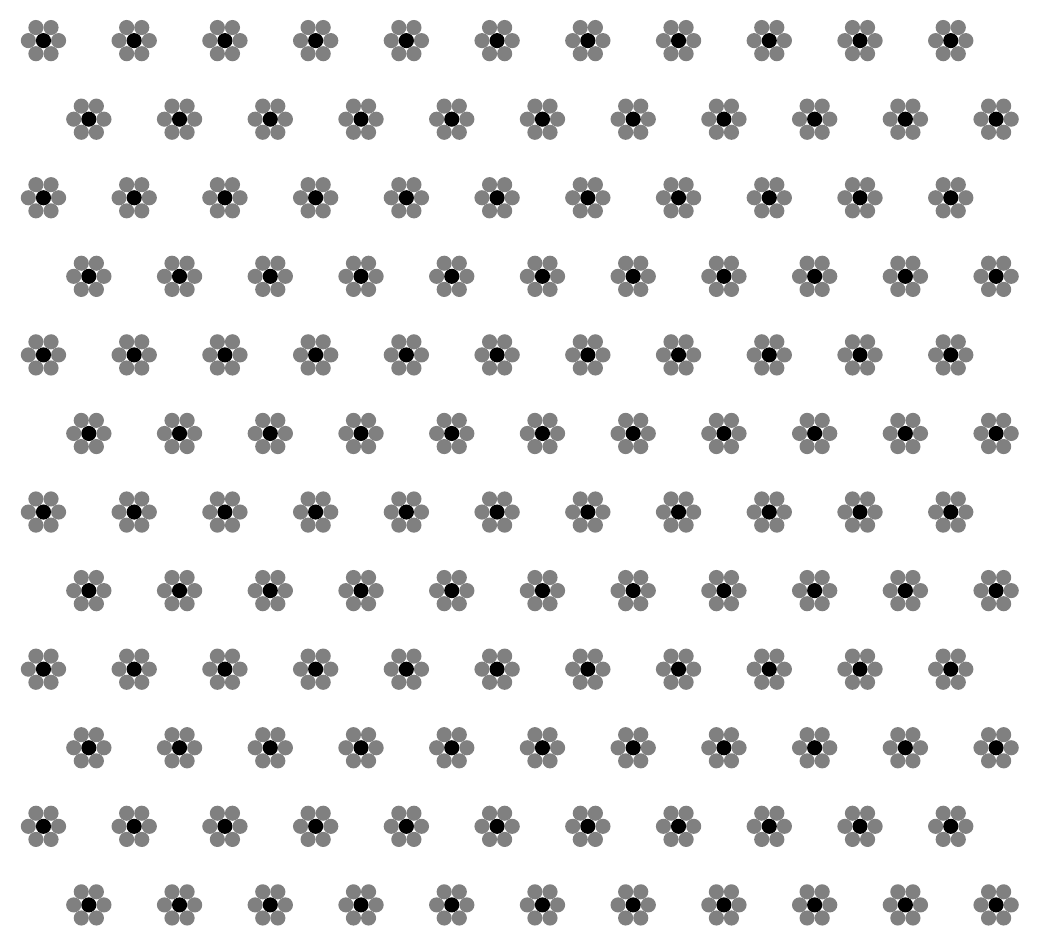}
\caption{Best configuration for  $d=2$ and $K=[1,2]\cup \{4\}$.}\label{fig:la4}
\end{figure}

As a proof of concept, we now prove that this linear program is sharp for $d=1$. In this case, we have the inequality $3\leq (2+\la)F(0)/\ft F(0)$. Define the following symmetric set
$$
E_\la=\bigcup_{n=-N}^{N}[3n-1/2,3n+1/2] 
$$
with $N=\lfloor (\la-1)/6\rfloor$. Then it is easy to see that $F_\la={\rm 1}_{E_\la}\star {\rm 1}_{E_\la} $ (where $\star$ is the convolution operator) is positive definite and $F_\la(x)=0$ if $|x|>\la$ or if $1<|x|<2$, since $E_\la+E_\la =\cup_{n=-2N}^{2N}[3n-1,3n+1] $. Moreover, $F_\la(0)/\ft F_\la(0)=1/\vol(E_\la)=1/(1+2N) \sim 3/\la$. We conclude that $\lim_{\la\to\infty} (2+\la)F_\la(0)/\ft F_\la(0) = 3$, proving the bound is sharp.

\subsection{Proof of Theorem \ref{thm:LP-bounds}.}
In view of Proposition \ref{prop:periodic_packings_approximate_Delta}, it suffices to consider just periodic sphere packings $P=Y+\Lambda+\bd $. Assume first $F$ is of Schwartz class. We apply Poisson summation to obtain
$$
(\#Y) F(0) \geq \sum_{y,y'\in Y}\sum_{\la \in \Lambda} F(\la+y-y') = \frac{1}{\vol(\r^d/\Lambda)}\sum_{\xi \in \Lambda^*} \ft F(\xi) \left| \sum_{y\in Y} e^{2\pi iy\cdot \xi}\right|^2 \geq \frac{(\#Y)^2\ft F(0)}{\vol(\r^d/\Lambda)}.
$$
Rearranging terms, we deduce that $\dens(P)\leq \vol(\bd)F(0)/\ft F(0)$. In the unconstrained packing problem one could now, by a standard convolution approximation argument, show the same inequality holds when $F$ is only continuous and $L^1$. However, this trick does not work for constrained sphere packings, since convolutions destroy $K$-admissibility of functions. Instead, we give below a direct proof based on a new trick involving the Féjer kernel.

Let $A$ be a full rank matrix such that $\Lambda=A\cdot\z^d$ and let $$f_\Lambda(x)=\displaystyle\sum_{\lambda\in\Lambda} F(x+\lambda).$$
Then we see that $f_\Lambda$ is $\Lambda-$periodic, that the summation above converges a.e. and that $f_\Lambda\in L^1(\r^d/\Lambda)$. Note also that $$f_\Lambda\circ A(x)=\displaystyle\sum_{k\in\z^d} F\circ A(x+k).$$
Observe that $f_\Lambda\circ A \in L^1(\r^d/\z^d)$ and $(f_\Lambda\circ A)^\wedge (\alpha)= |\det(A)|^{-1} \ft F\circ A^{-\top}(\al)$, where $\vol(\r^d/\Lambda)=|\det(A)|$ and we use $\top$ for transpose.
For $N$ a positive integer, consider the $N$-th Féjer kernel
\begin{equation*}
    \F_N(x)={\bf 1}_{Q_d}(x)\displaystyle\sum_{\substack{\al\in\z^d\\ |\al|_\infty<N}} \displaystyle\prod_{j=1}^{d} \left(1-\displaystyle\frac{|\al_j|}{N}\right)e^{2\pi i\al\cdot x} =  N^d \prod_{j=1}^d \left( \frac{\sin(\pi N x_j)}{N\sin(\pi x_j)} \right)^2 {\bf 1}_{|x_j|<1/2}
\end{equation*}
Denote by $\star$ the convolution operator. It is well known that the $\F_N$ is an approximate identity, that is, $\F_N\star G(x)\to G(x)$ as $N\to\infty$, for any $x$ and any bounded continuous function $G$. The nice bit is that if $G$ is $\z^d$-periodic, then convolution simply multiplies Fourier coefficients. Routine computations show that
$$
(f_\Lambda\circ A) \star \F_N(x) = \frac{1}{\vol(\r^d/\Lambda)}\sum_{\substack{\al\in\z^d\\ |\al|_\infty<N}} \ft F\circ A^{-\top}(\al) \displaystyle\prod_{j=1}^{d} \left(1-\displaystyle\frac{|\al_j|}{N}\right)  e^{2\pi i \al \cdot x},
$$
is indeed a trigonometric polynomial. We conclude that
\begin{align*}
\sum_{y,y'\in Y}  (f_\Lambda\circ A) \star \F_N(A^{-1}(y-y')) \geq \frac{(\#Y)^2}{\vol(\r^d/\Lambda)} \ft F(0)
\end{align*}
However, letting $m:=\max_{y\neq y'} {|y-y'|}+\sup(K)+\max_{x\in \r^d/\Lambda} |x|$, the left-hand side above is equal to
\begin{align*}
    &\sum_{y,y'\in Y}\displaystyle\int_{Q_d} \F_N(x)\left(f_\Lambda\circ A\right)\left(A^{-1}(y-y')-x\right)dx\\
     & = \frac{1}{|\det(A)|}\displaystyle\sum_{y,y'\in Y}\displaystyle\int_{\r^d/\Lambda} \F_N(A^{-1}x)f_\Lambda\left(y-y'- x\right)dx\\
    &\leq\displaystyle\frac{1}{\vol(\r^d/\Lambda)}\displaystyle\sum_{y,y'\in Y}\displaystyle\sum_{\substack{\lambda\in\Lambda\\
    |\lambda|\leq m}} \displaystyle\int_{\r^d/\Lambda} \F_N(A^{-1}x)F(y-y'-x+\lambda)dx\\
    &=\displaystyle\sum_{y,y'\in Y}\displaystyle\sum_{\substack{\lambda\in\Lambda\\
    |\lambda|\leq m}} \F_N\star(F\circ A)(A^{-1}(y-y'+\lambda))\\
    &\xrightarrow{N\rightarrow\infty} \displaystyle\sum_{y,y'\in Y}\displaystyle\sum_{\substack{\lambda\in\Lambda\\
    |\lambda|\leq m}} (F\circ A)(A^{-1}(y-y'+\lambda))\\
    &=\displaystyle\sum_{y,y'\in Y}\displaystyle\sum_{\substack{\lambda\in\Lambda\\
    |\lambda|\leq m}} F(y-y'+\lambda)\\
    &\leq \#Y\cdot F(0).
\end{align*}
The limit above is justified because we have a finite sum. The theorem follows. \qed

\section{One-dimensional packings and dominos}\label{sec:1D}

We now focus on the one-dimensional case and we ask for the existence of periodic packings with the maximal density. Let $I=[0,1]$ be the unit interval. We begin with the observation that we can restrict our attention to packings of the half-line $[0,\infty)$ only; that is,
$$
\Delta_1(K)=\sup_{P \in \P_1(K) \text{ and } P\subset [0,\infty)} \dens^+(P),
$$
where $\dens^+(P)=\limsup_{t\to\infty} \vol(P\cap [0,t])/t$.
Indeed, if $P\subset (-\infty,\infty)$ is a $K$-admissible periodic packing, then $$\dens(P) \leq \max\{\dens^+(P\cap \r_+),\dens^+(P \cap \r_-)\}.$$ Now note that for any $K$-admissible packing of $[0,\infty)$, we can assume that $I$ starts at $0$ and so the packing can be uniquely described as a sequence
$$
w=a_{1}a_{2}...a_{k}...,
$$
where $a_j\in K\cup (\sup(K),\infty)$  represent the distance between consecutive centers of intervals in the packing. From now on, we assume that $\sup(K)\in K$, and we set $N:=\lceil\sup(K)\rceil$. Hence, we can always assume that $a_j\in K$ for all $j$ since otherwise, we can reduce that distance without destroying $K$-admissibility and  obtain a denser packing. For simplicity, we just write $\pazocal{P}_+(K)$ for $K$-admissible packings of $[0,\infty)$ with distances between adjacent centers of intervals drawn from $K$ that start with $I$ touching $0$.

A \emph{word} will be a finite ordered sequence of elements of the alphabet $\Sigma=K$. For instance, $w=a_{1}a_{2}...a_{k}$ is a word of \emph{length} $\#w=k$. It will be useful to consider the \emph{empty} word, denoted by $\varnothing$, which has no elements and has length $\#\varnothing=0$. We can also define the \emph{norm} of $w$ by
$$
|w|=\sum_{i=1}^k a_i.
$$
We denote by $\Sigma^*$ the set of all finite words. We will now construct a domino set $\D$ by
\begin{align}
    \D_K:=\big\{(w,w')\in\Sigma^*\times\Sigma^*: & \#w=\#w'=N \text{ and any subword } {s} \text{ of }  ww' \\ & \text{ satisfies } |s| \in K \cup (\sup(K),\infty) \big\}.
\end{align}
Thus, $\D_K\subset \Sigma^* \times \Sigma^*$. Here, $w{w'}$ is the concatenated word. A linear domino tiling from $\D_K$ is an infinite word 
$$
\gamma=w_1w_2w_3...
$$
where $(w_i,w_{i+1})$ is a domino piece from $\D_K$. We let $\T(\D_K)$ be the set of domino tiling built this way. Given a tiling $\ga\in \T(\D_K)$, we let
$$
\dens(\ga) = \limsup_{t\to\infty} \frac{\#(w_1...w_t)}{|w_1....w_t|}.
$$
We say $\ga$ is periodic if for some $n$ we have $w_{i+n}=w_i$ for all $i$.

\begin{proposition}
    \label{prop:isomorphism_packings_and_tillings}
    There is a canonical bijection between $\P \in \pazocal{P}_+(K)\mapsto \ga_P \in \T(\D_K)$. This bijection maps periodic to periodic and satisfies
    $$
    \dens^+(P)=\dens(\ga_P).
    $$
\end{proposition}
\begin{proof}
    Indeed, given a packing $P\in\pazocal{P}_+(K)$, we can write $P$ as (infinite) sequence $a_1a_2...$, with $a_{j}\in K$. We then break this sequence into pieces of the form $w_{m+1}=a_{mN+1}...a_{(m+1)N}$, for $m\in\z_{\geq0}$, which are words belonging to $\Sigma^*$. By construction, the pairs $(w_m,w_{m+1})\in \D_K$, and then we can associate $P$ to the tiling $w_1w_2....\in\pazocal{T}(K)$. Clearly, this map is injective. To show it is surjective, let $\gamma=w_1w_2...\in\pazocal{T}(K)$ be any tiling. Then, by the choice of $N$, the packing $P$ whose sequence of distances between consecutive centers is given by the concatenation $w_1w_2w_3...$ belongs to $\pazocal{P}_+(K)$, and its image is the tiling $\gamma$. This establishes a bijection between $\P_+(K)$ and $\pazocal{T}$. Moreover, it is easy to check that $P$ is periodic if and only if $\ga_P$ is and that $\dens^+(P)=\dens(\ga_P)$ holds for any packing $P \in \P_+(K)$.
\end{proof}

In general, we define a linear domino game as a tuple $(\Sigma,\D)$, where $\Sigma$ is a set of symbols and $\D\subset \Sigma^*\times \Sigma^*$ are domino pieces,  where $\Sigma^*$ is the set of all finite words in the alphabet $\Sigma$. The set of linear domino tilings $\T(\D)$ are those infinite words such that $\gamma=w_1w_2w_3...$, where $(w_i,w_{i+1})\in \D$ for all $i$. We say that a map $f:\Sigma^* \to \r_+$ is a \emph{norm function} when it satisfies the following properties:

\begin{itemize}
\item[(a)] For any $w_1,w_2,w_3\in\Sigma^*$, the function $k \geq 0 \mapsto \frac{\#(w_1w_2^kw_3)}{f(w_1w_2^kw_3)}$ is monotone;
\item[(b)] The function $k>0\mapsto \frac{\#(w_2^k)}{f(w_2^k)}$ is a positive constant;
\item[(c)] For any $\epsilon>0$, there is $\delta>0$ such that if $w'$ is a sub-word of $w$ with $\#w-\#w'\leq\epsilon$, then $|f(w)-f(w')|\leq \delta$.
\end{itemize}
Note that if $f$ is a norm function, then
$$
\frac{\#(w_1w_2^kw_3)}{f(w_1w_2^kw_3)} = \frac{\#(w_2^k)+O(1)}{f(w_2^k)+O(1)} = \frac{\#(w_2^k)+O(1)}{\frac{f(w_2)}{\#w_2}\#(w_2^k)+O(1)} = \frac{\#w_2}{f(w_2)}+O(1/k)
$$
as $k\to\infty$. For a norm function $f$ and a linear domino tiling $\gamma=w_1w_2w_3...\in\T(\D)$, we let
$$
\dens_f(\ga)=\limsup_{t\to\infty} \frac{\#(w_1....w_t)}{f(w_1....w_t)}.
$$
In particular, $\dens_f(\be\be\be...)=\frac{\#\be}{f(\be)}$. This is a generalization of the scenario described before by taking $f=|\cdot|$ and $\D=\D_K$.

\begin{theorem}\label{thm:finitedomino}
Let $\Sigma$ be a set of symbols and $\D\subset \Sigma^*\times \Sigma^*$ be a finite domino set. Let $f:\Sigma^* \to \r_+$ be a norm function. Assume that $\sup_{\ga\in \T(\D)} \dens_f(\ga) \in (0,\infty)$. Then there exists a periodic linear domino tiling $\ga \in \T(\D)$, with period $\leq \#\D$, such that $$\dens_f(\ga)=\sup_{\al\in\T(\D)} \dens_f(\al).$$
\end{theorem}

The following corollary is an immediate consequence of the above theorem by using $\Sigma=K$ and $f=|\cdot|$.

\begin{corollary}
\label{existence_of_maximal_periodic_packing_finite_case}
    If $K$ is finite, then there exists a periodic $K$-admissible sphere packing of $\r$ of maximal density.
\end{corollary}

In order to prove Theorem \ref{thm:finitedomino}, we need to introduce some terminology. Let $V:= \{w\in \Sigma^* : \exists w'\  s.t. \ (w,w')\in \D \text{ or } (w',w)\in \D\}$. We then can define a directed graph by $\G=(V,\D)$, with vertex set $V$ and directed edges $\D$. From now on, in $\G$, we only consider directed paths (i.e., finite sequences of the form $\ga=w_1...w_m$ such that $(w_j,w_{j+1}) \in \D$ for any $j$). We say that the \emph{path length} of $\ga$ is $m$. We also say that $\ga$ is a \emph{closed path} when $(w_m,w_1)\in\D$ and \emph{open} otherwise. An atom is a simple closed path (closed and with no self-intersections). 

\begin{proof}[Proof of Theorem \ref{thm:finitedomino}]
Since $\D$ is finite, then $V$ is finite and the set of atoms $\A$ is also finite. First, we claim that for any closed path $\beta=w_1...w_m$ in $\G$, it holds that

\begin{equation}
\label{eq:inequality_density_closedpath_atoms}
\frac{\#\beta}{f(\beta)} \leq \max_{\al\in\A} \frac{\#\al}{f(\al)}.
\end{equation}
The inequality is trivial if $\beta$ is an atom. Assume by induction the claim is true if the path length of $\be$ is $\leq M$. For $m=M+1$ and $\beta$ not an atom, then $\be$ must have some self-intersection (i.e., there exist $1\leq i<j\leq M+1$ such that $w_i=w_j$). Let
$v_1=w_1...w_{i-1}$, 
$v_2=w_i...w_{j-1}$ and
$v_3=w_j...w_{M+1}$. Then $v_1,v_2,v_3\in\Sigma^*$, $v_1v_3$ and $v_2$ are closed paths in $\G$ with path lengths $\leq M$ (note that we may have $v_1$ or $v_3$ the empty word, but not both at the same time). By monotonicity hypothesis and the induction hypothesis, we deduce
\begin{equation}
\frac{\#\beta}{f(\beta)}=\frac{\#(v_1v_2v_3)}{f(v_1v_2v_3)}\leq \max\left\{\frac{\#(v_1v_3)}{f(v_1v_3)},\frac{\#v_2}{f(v_2)}\right\}\leq \max_{\al\in\A} \frac{\#\al}{f(\al)}
\end{equation}
as desired. Let now $\ep>0$ be sufficiently small and $\ga=w_1w_2w_3...\in\T(\D)$ be such that $\dens_f(\ga)+\ep> \sup_{\al\in\T(\D)} \dens_f(\al)>\ep$. Let $m_n\uparrow\infty$ be a sequence such that $$\dens_f(\ga)=\lim_{n\rightarrow\infty} \frac{\#(w_1...w_{m_n})}{f(w_1...w_{m_n})}.$$ Since $\#\D$ is finite, there exists a word $w$ which appears infinitely many times in the sequence of $\ga$. We may assume such word is $w_1$ since the properties of $f$ imply that
$$
\lim_{n\rightarrow\infty} \frac{\#(w_1...w_{m_n})}{f(w_1...w_{m_n})} = \lim_{n\rightarrow\infty} \frac{\#(w_i...w_{m_n})}{f(w_i...w_{m_n})}
$$
for any $i$. For each $n$, we fix a path $\ga_{(w_{m_n},w)}$ in $\G$ from $w_{m_n}$ to $w$ (such a path exists since $w$ appears infinitely many times in $\ga$). By removing atoms, we may assume that $\ga_{(w_{m_n},w)}$ is simple, so that $\ga_{(w_{m_n},w)}$ has path length at most the cardinality of $\D$. Since $\D$ is finite, we have just finitely many simple paths $\ga_{(w_{m_n},w)}$, so $\#\ga_{(w_{m_n},w)}=O(1)$. Let $\ga_{(w_{m_n},w)}^*$ be the path $\ga_{(w_{m_n},w)}$ without the endpoints $w_{m_n}$ and $w$, so that $\ga_{(w_{m_n},w)}=w_{m_n}\ga_{(w_{m_n},w)}^*w_1$. Hence, $w_1...w_{m_n}\ga_{(w_{m_n},w)}^*$ is a closed path. We can now use the properties of $f$ to conclude
\begin{align*}
\lim_{n\to\infty} \frac{\#(w_1...w_{m_n}\ga_{(w_{m_n},w)}^*)}{f(w_1...w_{m_n}\ga_{(w_{m_n},w)}^*)} = \lim_{n\to\infty} \frac{\#(w_1...w_{m_n})}{f(w_1...w_{m_n})}=\dens_f(\ga).
\end{align*}
However, $\frac{\#(w_1...w_{m_n}\ga_{(w_{m_n},w)}^*)}{f(w_1...w_{m_n}\ga_{(w_{m_n},w)}^*)}\leq \max_{\al\in\A} \frac{\#\al}{f(\al)}$.  This concludes the proof.
\end{proof}

\begin{figure}[h]
\vspace{-.1cm}
\begin{center}
\includegraphics[scale=.5]{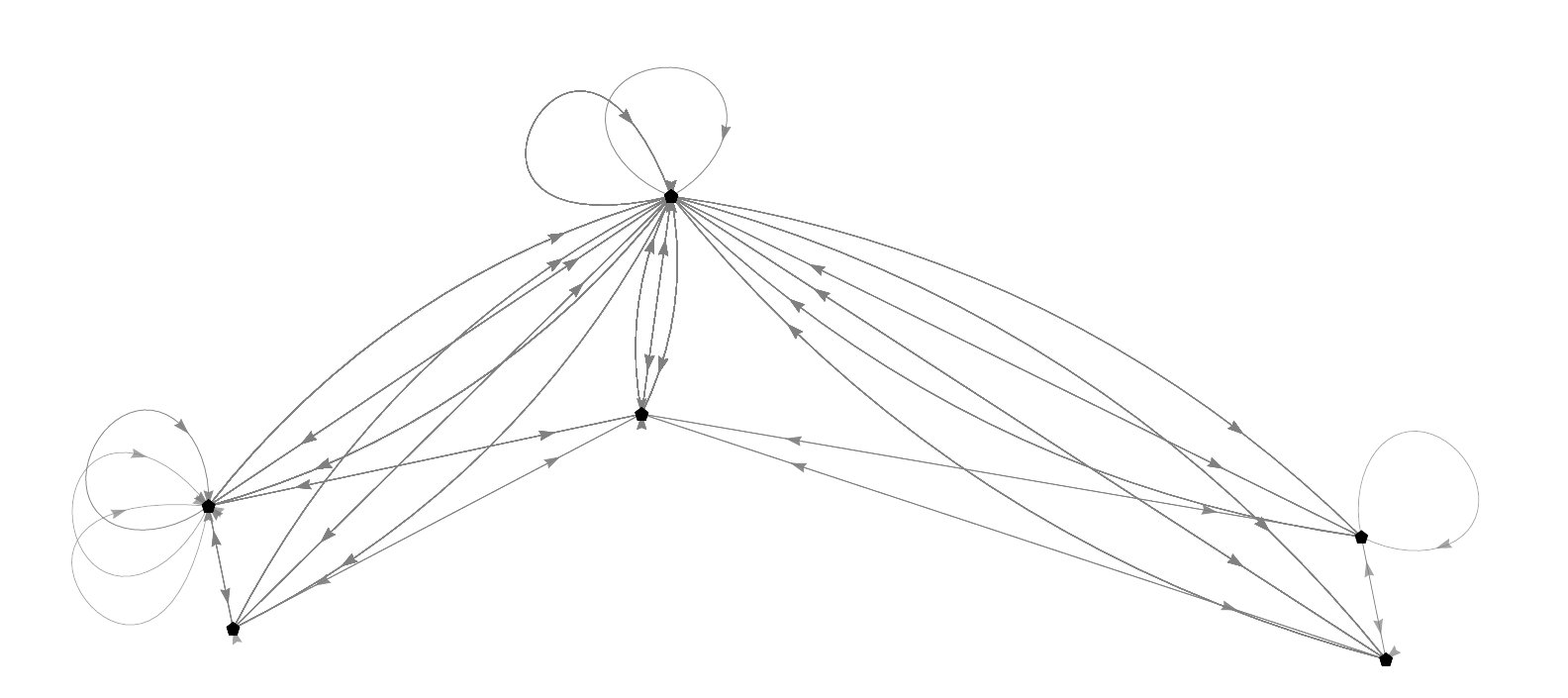}
\end{center}
\vspace{-.5cm}
\caption{A visualization of the graph $\G$ for $K=\{1,2,\be\}$ and $3<\be \leq 4$ via higher-dimensional embedding.}
\end{figure}

The proof of Theorem \ref{thm:finitedomino} does not work for an arbitrary compact $1\in K \subset [1,\infty)$ since it strongly relied on the fact that the graph $\G$ (or the domino set $\D$) was finite, so that any closed path could be decomposed into atoms, and hence, in this case, the atoms control the value $\Delta_1(K)$. Nevertheless, our proof describes an algorithm to find optimal $K$-admissible sphere packings in the case $K$ is finite.

We now analyse the case when $K$ has 3 elements.

\begin{lemma}\label{thm:Kalbe}
Let $1<\al<\be$ and $K=\{1,\al,\be\}$ with $1<\al<\be$. If $\beta\leq 2$, the periodic packing $1111...$ has maximal density. If $\beta>2$, we have the following cases:

\begin{center}
    \begin{tabular}{c|c|c}
        & conditions on $\al,\beta$ & a periodic packing of maximal density \\ \hline
        1 & $\al=2$ and $\beta\leq3$ & $1111...$ \\ \hline
        2 & $\al=2$ and $3<\beta$ & $11\beta 11\beta...$ \\ \hline
        3 & $\al\neq2$ and $\beta\leq 1+\al$ & $1\al1\al...$ \\ \hline
        4 & $\al\neq2$, $1+\al<\beta\leq 2\al$ and $2\al\leq \beta+1$ & $\al\al...$ (and $1\beta 1\beta...$ if $2\al=\beta+1$)  \\ \hline
        5 & $\al\neq2$, $1+\al<\beta\leq 2\al$ and $2\al> \beta+1$ & $1\beta1\beta...$ \\ \hline        
        6 & $\al\neq2$ and $2\al< \beta$ & $1\beta1\beta...$ \\ \hline 
    \end{tabular}
\end{center}

\end{lemma}

\begin{proof}[Proof of Lemma \ref{thm:Kalbe}.]
Let $P\in\P_+(K)$ be a $K$-admissible periodic packing and let $\ga$ be its sequence of distances. We have two cases: the one in which $\beta$ does not appear in $\ga$ and the second one in which $\beta$ appears infinitely many times in $\ga$. In the second case, using the graph language, we break $\ga$ into simple closed paths with endpoints $\beta$, and we try maximize the density of such paths.


\noindent \textbf{Case 1.} It is trivial.


\noindent \textbf{Case 2.} If $\beta\leq 4$, in the first case ($\beta$ does not appear in $\ga$), we must have $\ga=222...$ whose density is 1/2. In the second case, the densest path is $\beta 11$, with density $3/(2+\beta)$ which gives the packing $\beta 11\beta 11...$. Since $3/(2+\beta)\geq1/2$, the result follows. Now, if $\beta>4$, for the first case, there is no periodic packing, and for the second one, the densest path is again $\beta11$.


\noindent \textbf{Case 3.} For the first case, $1\al 1\al...$ is the densest possibility. For the second, the paths $\beta (1\al 1)^k$ for $k\rightarrow\infty$ have increasing density, and they converge to $1\al 1\al...$, which is the densest periodic packing in this case.


\noindent \textbf{Case 4.} In this case, the pairs $(1,1)$ and $(1,\al)$ cannot be consecutive distances. For the first case, we case the sequence $\al\al...$ with density $1/\al$. For the second case, we have the paths $\beta\al^k$ for $k\geq1$ and $\beta1$. The densest packing is $\al\al...$. If $2\al=\beta+1$, then $1\beta 1\beta...$ has also maximal density.


\noindent \textbf{Case 5.} By the same arguments as in Case 4, we deduce that the densest packing is $1\beta 1\beta...$.


\noindent \textbf{Case 6.} We have to consider just the second case. Since the pairs $(1,1)$, $(1,\al)$ and $(\al,\al)$ cannot be consecutive distances, the densest path is $\beta 1$, which provides the packing $\beta 1\beta 1...$.
\end{proof}

\section{Proof of Theorem \ref{thm:1D}}
Let $K'=\{\al_1,...,\al_N\}$ the set of limit points of $K$. We can then write $K$ as follows: 
\begin{equation}
    K=\ft{K}\cup\displaystyle\bigcup_{j=1}^{N} \left\{\al_j+\lambda_k^j : k\geq1\right\},
\end{equation}
where $\ft{K}:=\left\{\al_1,...,\al_N,\beta_1,...,\beta_M\right\}$, with $1\in\ft{K}$, and such that
\begin{itemize}
    \item[(i)] $\al_i+\al_j\not\in\left\{\al_r : 1\leq r\leq N\right\};$ 
    \item[(ii)] $0<\lambda_k^j<\frac{\delta}{2}$ for any $j,k$, where $\delta:=\displaystyle\min_{1\leq i<j\leq N} \left|\al_i-\al_j\right|$;
    
    \item[(iii)] For any fixed $j$, we have $\lambda_k^j\downarrow0$ when $k\rightarrow\infty$;
    \item[(iv)] $\beta_M=\max{K}$.
\end{itemize}
We will show that under such conditions, for every packing $P\in\pazocal{P}_+(K)$, one can construct another one $\wt P$ with $\dens^+(\wt P)\geq \dens^+(P)$ and such that the set of distances of $\wt P$ lies in a finite set $\wt K \subset K$. Thus, by Lemma \ref{lem:enumerableK} and Corollary \ref{existence_of_maximal_periodic_packing_finite_case}, there exists a $K$-admissible periodic sphere packing of maximal density.
 
For $\gamma\in K$ and $1\leq i,r\leq N$, we set
\begin{align*}
    J\left(\gamma,\al_i,r\right)  :=\left\{s\geq 1 :  \gamma+\al_i+\lambda_s^r\in K\cup\left(\beta_M,\infty\right)\right\} \ \text{ and } \
    M\left(\gamma,\al_i,r\right)  :=\displaystyle\sup{J\left(\gamma,\al_i,r\right)}.
\end{align*}
Observe that, for fixed $i,j$, we may have $\al_j+\lambda_k^j+\al_i=\al_w$ for some $w$. In this case, the index $k$ with such property is unique because of (ii). Therefore, we may write $k(j,i)$ to denote precisely this index $k$, whenever it exists. Observe also that if $M\left(\gamma,\al_i,r\right)=\infty$, then $\ga+\al_i$ is an accumulation point and either $\gamma+\al_i=\al_w$ for some $w$ or $\gamma+\al_i\geq\beta_M$. In particular, if $k\neq k(i,j)$, then $M\left(\al_j+\lambda_k^j,\al_i,r\right)=\infty$ if and only if $\al_j+\lambda_k^j+\al_i\geq\beta_M$. Our first step is to ask how large $M(\ga,\al_i,r)$ can be when it is finite.

\noindent \textbf{Step 1.} We claim that there is $C>0$ such that for all $1\leq i,j,r\leq N$ and $k\geq 1$, with $k\neq k(i,j)$, we have $M\left(\al_j+\lambda_k^j,\al_i,r\right)\leq C$ whenever it is finite. Indeed, assume by contradiction this does not happen. Then, for each $C\in \z_+$ there is $s=s(C)\geq C$, $k\geq 1$ ($k\neq k(i,j)$) and $1\leq i,j,r\leq N$ such that $\al_j+\la^j_k+\al_i+\la^r_s \in K\cup (\be_M,\infty)$ but $M\left(\al_j+\lambda_k^j,\al_i,r\right)<\infty$. Thus, as $s\to\infty$, some triple $(i,j,r)$ is repeated infinitely often, and we choose such triple. If $k=k(C)=O(1)$, then for some $k_0\neq k(i,j)$, we have that $\al_j+\la^j_{k_0}+\al_i \in K\cup (\be_M,\infty)$ is an accumulation point and $M\left(\al_j+\lambda_{k_0}^j,\al_i,r\right)<\infty$. Since $k_0\neq k(i,j)$, we have that $\al_j+\la^j_{k_0}+\al_i \geq \be_M$, but this is absurd since we would have that $M\left(\al_j+\lambda_{k_0}^j,\al_i,r\right)=\infty$, which is not the case. Thus, $k=k(C)$ is unbounded, and so $\al_j+\al_i \in K\cup (\be_M,\infty)$ is an accumulation point. Condition (i) implies that $\al_j+\al_i \geq \be_M$, and again we reach a contradiction since this implies that $M\left(\al_j+\lambda_{k}^j,\al_i,r\right)=\infty$ for all $k$. 

\noindent \textbf{Step 2.} Let $C>0$ be the constant from the previous step. Enlarge $C$ if necessary so that we also have
$$
C\geq M(\ga,\al_i,r)
$$
for all $\ga\in \ft K$ and $1\leq i,r\leq N$ such that $M(\ga,\al_i,r)<\infty$. 
Define the set $\wt{K}$ by
\begin{equation}
    \wt{K}:=\ft{K}\cup\displaystyle\bigcup_{j=1}^{N} \left\{\al_j+\lambda_k^j; \text{ } 1\leq k\leq C\right\} \subset K.
\end{equation}
Let $P=d_1d_2d_3...\in\pazocal{P}_+(K)$ be a packing where $d_i\in K$. By an inductive argument, we will replace $d_i$ by $\wt{d_i}\leq d_i$, with $\wt{d_i}\in\wt{K}$, in such a way that $\wt P=\wt d_1 \wt d_2 ... \in \P_+(\wt K) \subset \P_+(K)$. This clearly implies that $\dens^+(\wt P)\geq \dens^+(P)$.

We start with the base case. If $d_1\in\wt{K}$, there is nothing to prove. Assume we have $d_1\not\in\wt{K}$; hence, $d_1\not\in\ft{K}$. Therefore, we have $d_1=\al_i+\lambda_s^i$ for some $1\leq i\leq N$ and $s>C$. We claim that $M(d_2,\al_i,i)=\infty$. Indeed, if it were finite, then $M(d_2,\al_i,i)\leq C$, and hence, $s\leq C$, a contradiction. Therefore, $M(d_2,\al_i,i)=\infty$; hence, $d_2+\al_i+\lambda_t^i\in K\cup\left(\beta_M,\infty\right)$ for infinitely many indices $t$. Therefore, we must have $d_2+\al_i\in K\cup\left(\beta_M,\infty\right)$. We now claim that we can replace $d_1=\al_i+\lambda_s^i$ by $\wt{d_1}:=\al_i$. In other words, we must verify that $\wt{d_1}$ is compatible with $d_2,...,d_{1+\left\lceil\beta_M\right\rceil}$. This is the reason why we considered the more general sets $J(\gamma,\al_i,r)$ with $r\neq i$. Since $d_2+\al_i\in K\cup\left(\beta_M,\infty\right)$, it follows that $\wt{d_1}$ is compatible with $d_2$.  If $\al_i+d_2\geq\beta_M$, then $\al_i+d_2+...+d_w\geq\beta_M$ for any $w\geq3$; hence, $\wt{d_1}$ satisfies all the required compatibility conditions. Assume we have $\al_i+d_2<\beta_M$. Since $\al_i+d_2+\lambda_t^i\in K\cup\left(\beta_M,\infty\right)$ for infinitely many $t$, it must be the case in which $\al_i+d_2=\al_r$ for some $r$. We claim that $M(d_3,\al_r,i)=\infty$. Indeed, if it were $<\infty$, then $M(d_3,\al_r,i)\leq C$, and, hence $s\leq C$, a contradiction. Therefore, we have $\al_r+d_3+\lambda_t^i\in K\cup\left(\beta_M,\infty\right)$ for infinitely many $t$, and hence, $\al_i+d_2+d_3=\al_r+d_3\in K\cup\left(\beta_M,\infty\right)$ which proves $\wt{d_1}$ is also compatible with $d_3$. If $\al_r+d_3\geq\beta_M$, then we are done. Otherwise, we apply the same procedure to show that $\wt{d_1}$ is compatible with $d_4$, and so on. By repeating this process at most $\left\lceil\beta_M\right\rceil$ times, we conclude that, indeed, it is possible to replace $d_1$ by $\wt{d_1}=\al_i$.

Assume that we have replaced $d_1,...,d_n$ by $\wt{d_1},...,\wt{d_n}\in\wt{K}$. If $d_{n+1}\in\wt{K}$, then we are done. If not, we repeat the same procedure as for $d_1$, but now for both the right- and the left-hand side of $d_{n+1}$. This completes the induction argument. Observe that $d_n\neq\wt{d_n}$ precisely when $d_n=\al_i+\lambda_k^i$, for $k>C$, and in this case, we took $\wt{d_n}=\al_i<d_n$. This process does not reduce the density of the packing.
\hfill \qed

\begin{obs}
The main idea of this proof was to take a packing of maximal density and reduce to the finite case by replacing the distances close enough to an accumulation point by this point. Such technique relies a lot on the structure of the set $K$ (existence of left or right limit points and the arithmetic structure of $K'$). As the complexity of $K$ grows, such technique becomes very hard to implement. Nevertheless, the above proof gives an algorithm to find the best packing in such a case.
\end{obs}

\section{Constructions via modular forms}

In this section, we present some constructions via modular forms that generalize the ones in \cite{CG19,CKMRV17,Vi17} to all dimensions divisible by $4$.

Throughout the rest of the paper, we will always use $z$ for a variable in the upper-half plane $\HH=\{x+it: x\in \r, t>0\}$.  We will use the convention $q=e^{2\pi i z}$ and $r=e^{\pi i z}$ for $z\in \HH$. We will be handling holomorphic modular forms $f:\HH\to\cp$ over the principal congruence subgroups
$$
\Gamma(N) = \{\ga\in {\rm SL}_2(\z) : \ga \equiv  \text{Id} \ (\text{mod } N) \}
$$
(note $\Gamma(1)={\rm SL}_2(\z)$). A holomorphic modular form of weight $k$ for a subgroup $\Gamma<\Gamma(1)$ is a holomorphic function $f: \HH \to \cp$ invariant by the slash operation
$$
f|_k\ga(z) :=  (cz+d)^{-k}f\bigg(\frac{az+b}{cz+d}\bigg)=f(z), \ \ \text{ for all } z\in \HH  \text{ and } \ga=\left(\begin{matrix} a & b\\ c & d\end{matrix}\right)\in \Gamma,
$$
and such that 
$$
f|_k\ga(z) = O(1) \ \ \text{ as } \im z\to\infty.
$$
We denote by $M_k(\Gamma)$  for the space of holomorphic modular forms $f:\HH\to\cp$ of weight $k$ for $\Gamma$. We simply write $M_k$ for $M_k(\text{SL}_2(\z))$. We let
$$
T = \left( \begin{matrix} 1 & 1\\ 0 & 1\end{matrix}\right) \ \ \text{and} \ \ S = \left( \begin{matrix} 0 & -1\\ 1 & 0\end{matrix}\right).
$$
If $T^2 \in \Gamma$, then any $f\in M_k(\Gamma)$ has a Fourier $r$-series of the form
$$
f(z) = \sum_{n\geq 0} a_n r^n.
$$
The Eisenstein series
\begin{align*}
E_4(z)  = 1+240 \sum_{n\geq 1} \si_{3}(n)q^{n} \ \ \text{ and } \ \  E_6(z)  =1- 504 \sum_{n\geq 1} \si_{5}(n)q^{n},
\end{align*}
are classical examples of modular forms in $M_4$ and $M_6$, respectively. It is easy to see that $M_k$ is trivial if $k$ is odd or if $k<4$ is nonzero, while $M_0$ contains only constants.
Indeed, a fundamental result is that
$$
\bigoplus_{j\in \z} M_{j} = \cp[E_4,E_6].
$$
We will also need Ramanujan’s cusp form of weight $12$
$$
\Delta(z)= \frac{E_4^3(z)-E_6^2(z)}{1728} = q \prod_{n\geq 1}(1-q^n)^{24} \in M_{12},
$$
which clearly never vanishes for $z\in \HH$. We will also need the Jacobi theta functions defined by
\begin{align*}
\Theta_{00}(z) = \sum_{n \in \z} r^{n^2}, \quad \Theta_{10}(z)  = \sum_{n \in \z} r^{ (n+1/2)^2}\quad \text{and} \quad \Theta_{01}(z) = \sum_{n \in \z} (-1)^n r^{n^2}.
\end{align*}
We define their fourth powers by
$$
U=\Theta_{00}^4, \quad V=\Theta_{10}^4\quad \text{and} \quad W=\Theta_{01}^4.
$$ 
These are modular forms of weight $2$ that satisfy the Jacobi identity 
$$
U=V+W
$$ and the transformation laws
\begin{align*}
& U|_2T = W, \ \ V|_2T = -V, \ \ W|_2T = U,\\
& U|_2S = -U, \ \ V|_2S = -W, \ \  W|_2S = -V.
\end{align*}
The functions $U,V,W$ are examples of modular forms of weight $2$ for $\Gamma(2)$. Another fundamental result is that
$$
\bigoplus_{j\in \z} M_{j}(\Gamma(2)) = \cp[U,V,W].
$$
Indeed, $M_k(\Gamma(2))$, for even $k\geq 2$, coincides with the space of homogeneous polynomials of degree $k/2$ in any two of the $U,V,W$. Finally, we let 
$$
E_2(z) = 1-24 \sum_{n\geq 1} \si_{1}(n)q^{n}
$$
be the quasimodular Eisenstein series of weight $2$. It satisfies the transformation rules 
$$
E_2(z+1)=E_2(z) \text{ and } E_2(-1/z)z^{-2} = E_2(z) + \frac{6}{\pi i z}.
$$
We define the space of holomorphic quasimodular forms of weight $k$ and depth $p$ over a subgroup $\Gamma$ by
\begin{align}\label{eq:quasimodularcharac}
M_k^{\leq p}(\Gamma) = \bigoplus_{j=0}^p E_2^jM_{k-2j}(\Gamma).
\end{align}
We again omit $\Gamma$ when $\Gamma={\rm SL}_2(\z)$. For more details about all these modular forms we recommend \cite{Z}.

The following proposition is key to build admissible functions for the linear programming bounds we have developed. 
\begin{proposition}\label{prop:Fconstruct}
Let $\Lambda\subset \r$ be finite, with $0\in \Lambda$ and such that $-1/\la\in \Lambda$ whenever $\la\in \Lambda$ and $\la\neq 0$. Let 
$$
p(s) = \sum_{\la\in \Lambda} a_\la e^{\pi i s \la}
$$
be a trigonometric polynomial. Let $\ep \in \{-1,1\}$, $d\geq 1$ be an integer and $f:\HH\to \cp$ be analytic. Suppose that
\begin{enumerate}
\item[(a)] We have
$$
\int_0^1|f(it)|t^{-d/2} \d t <\infty;
$$
\item[(b)] There is $\delta>0$ such that for any $c>0$, we have
$$
f(z)=O_c(e^{\pi \delta \Im z})
$$
if $\Im z > c > |\Re z| $;
\item[(c)] For all $\la\in \Lambda\setminus \{0\}$ and $z\in \HH$, we have
$$
a_\la f(z-\la) = -\ep a_{-1/\la} (z/i)^{d/2-2}f(-1/z+1/\la);
$$
\item[(d)] For all $z\in \HH$, we have
$$
\sum_{\la \in \Lambda} a_\la f(z-\la) = \ep a_0 (z/i)^{d/2-2} f(-1/z).
$$
\end{enumerate}
Then the function
$$
h(s) = p(s)\int_{0}^{\infty i} f(z)e^{\pi i z s} \d z,
$$
defines an analytic function for $\Re s >0$ that extends to an continuous function in $\re s\geq 0$, and satisfies the identity
\begin{align}
\begin{split}\label{eq:h-id}
& h(s) \\ & = \sum_{\substack{\la\in \Lambda \\ \la >0}} \int_\la^i a_\la f(z-\la)(e^{\pi i z s } +\ep (i/z)^{d/2}e^{\pi i (-1/z) s})\d z + a_0\int_0^i f(z)(e^{\pi i z s } + \ep(i/z)^{d/2}e^{\pi i (-1/z) s})\d z.
\end{split}
\end{align} 
if $\re s \geq 0$. In particular, the function
$
F(x)=h(|x|^2)
$
belongs to $L^1(\r^d) \cap C^\infty(\r^d)$ and $\ft F(x) = \ep F(x)$.
\end{proposition}
\begin{proof}
Conditions (a) and (b) guarantee that $h(s)$ is analytic in the region $\Re s>\delta$. The analyticity of $h(s)$ for $0<\re s <2\delta$ and continuity of $h(s)$ for $\re s=0$ follows straightforwardly by identity \eqref{eq:h-id} and condition (b). The fact that $F(x)$ defines a radial $L^1$-function and $\ft F(x) = \ep F(x)$ follows by identity \eqref{eq:h-id}, condition (a) and the fact that $x\in \r^d \to (e^{\pi i z |x|^2 } +\ep (i/z)^{d/2}e^{\pi i (-1/z) |x|^2})$ is a eigenfunction of the Fourier transform in $\r^d$ with eigenvalue $\ep$. It remains to prove identity \eqref{eq:h-id}. By analytic continuation (and condition (b)), it is enough to prove it for $\Re s> \delta$. We then have
\begin{align*}
h(s) & = \sum_{\la \in \Lambda} a_\la \int_{\la}^{\la+\infty i} f(z-\la)e^{\pi i z s}\d z \\
& = \sum_{\substack{\la\in \Lambda \\ \la \neq 0}} a_\la \int_{\la}^{i} f(z-\la)e^{\pi i z s}\d z + \int_{i}^{\infty i} \big(\sum_{\la \in \Lambda} a_\la  f(z-\la)\big) e^{\pi i z s} \d z + a_0\int_0^i f(z) e^{\pi i z s} \d z
\end{align*}
where in the second line above, we applied condition (b) to split the integral from $\la$ to $i$ and $i$ to $\infty i$. We obtain
\begin{align*}
& h(s) \\ &  = \sum_{\substack{\la\in \Lambda \\ \la > 0}} \bigg(a_\la \int_{\la}^{i} f(z-\la)e^{\pi i z s}\d z + a_{-1/\la} \int_{-1/\la}^{i} f(z+1/\la)e^{\pi i z s}\d z\bigg) \\ & \quad + a_0\int_{i}^{\infty i} \ep (z/i)^{d/2-2} f(-1/z) e^{\pi i z s} \d z + a_0\int_0^i f(z) e^{\pi i z s} \d z \\ & = \sum_{\substack{\la\in \Lambda \\ \la > 0}} \int_{\la}^{i}  \big(a_\la f(z-\la)e^{\pi i z s}\d z + a_{-1/\la}  f(-1/z+1/\la)e^{\pi i (-1/z) s}z^{-2}\big)\d z \\ & \quad + a_0\int_{0}^{i}\big(\ep (i/z)^{d/2}  e^{\pi i (-1/z) s} + e^{\pi i z s}\big)f(z) \d z
\\ & = \sum_{\substack{\la\in \Lambda \\ \la > 0}} \int_{\la}^{i}  a_\la f(z-\la)\big(e^{\pi i z s} + \ep(i/z)^{d/2}e^{\pi i (-1/z) s}\big)\d z  + a_0\int_{0}^{i}\big(\ep (i/z)^{d/2}  e^{\pi i (-1/z) s} + e^{\pi i z s}\big)f(z) \d z,
\end{align*}
where in the first equality, we have applied condition (d) and that $\Lambda$ is closed by $\la\to -1/\la$, in the second equality, we used the change of variables $z\to-1/z$, and in the last, we have applied condition (c).
\end{proof}

\begin{obs}
It is straightforward to show that if we strengthen condition (a) to
$$
\int_0^1|f(it)|t^{-k} \d t <\infty;
$$
for all $k>0$, then $\partial_{x}^\al F \in L^1(\r^d)$ for all multi-indexes $\al\in \z_+^d$, and since $\ft F=\ep F$, we conclude that $F$ is of Schwartz class.
\end{obs}

This integral transform was used by Viazovska et al. in \cite{CKMRV17,Vi17} for $p(s)=\sin^2(\pi s/2)$ to build the magic functions for the packing problem, and also by Radchenko and Viazovska in \cite{RV19}, for $p(s)=\sin(\pi s)$, to construct Fourier interpolation formulae with $\sqrt{n}$-nodes. For our purposes, we wish to investigate the case when $p(s)=\sin^2(\pi s/2)$. The following lemmas generalize the constructions in \cite{CKMRV17,Vi17} to all dimensions multiple of $4$.

\begin{obs}
To make computations simpler we will, from now on, adopt the nonstandard notation that
$$
f_T(z) =f|_kT(z)=f(z+1) \text{ and } f_S(z)=f|_kS(z)=f(-1/z)z^{-k},
$$
whenever it can be inferred from the context that $f$ has weight $k$; that is, $f\in M_k^{\leq p}(\Gamma)$, for some $\Gamma$, $k$ and $p$.
\end{obs}

\begin{lemma}[$\Gamma(2)$-construction]\label{lem:Ga2construction}
Let $d\geq 4$ be an integer divisible by $4$. Let $l\geq d/12$ be an even integer and set 
$$
k=2-d/2+6l  \text{ and }   \ep=(-1)^{d/4+1}.
$$
Consider the following linear operator \footnote{We recall the bracket notation for the coefficient of a power series: if $f(x)=\sum_{n\geq0} a_n x^n$, then $[x^n]f=a_n$ is the coefficient of the term $x^n$.}
\begin{align*}
L: \p \in M_k(\Gamma(2)) & \mapsto (\p_T+\p_S - \p,\,  [r^0] \p_S,\, [r]\p_S, ..., \, [r^l]\p_S).
\end{align*}
Then any function in the vector space 
$$
\F_{d,l}:=\frac{1}{\Delta^{l/2}}\ker(L)
$$
satisfies all conditions in Proposition \ref{prop:Fconstruct} for $\ep$, $d$ and $p(s)=\sin^2(\tfrac{\pi}{2} s)$. In particular, for any $f\in \F_{d,l}$, we have an associated radial Schwartz function $F_f:\r^d\to\cp$ such that $\ft {F_f}=\ep F_f$ and
$$
F_f(x) = \sin^2(\tfrac{\pi}{2} |x|^2)\int_0^{\infty i} f(z)e^{\pi i z |x|^2}\frac{\d z}{i}
$$
if $|x|^2>l$.
\end{lemma}
\begin{proof}
We will show that all conditions in Proposition \ref{prop:Fconstruct} are satisfied for $p(s)=\sin^2(\tfrac{\pi}{2}s)$. First, observe that if $f\in \F_{d,l}$, then condition (a) in Proposition \ref{prop:Fconstruct} is equivalent to
$$
\int_{i}^{\infty i} |f|_kS(z)||z|^{6l} |\d z| <\infty.
$$
This is true because $f=\p/\Delta^{l/2}$ and
$$
z^{6l}f|_kS =\frac{\p_S(z)}{\Delta^{l/2}(z) },
$$
and so we conclude the Fourier expansion of $z^{6l}f|_kS$ starts at $r$. Condition (b) in Proposition \ref{prop:Fconstruct} is trivial since $f$ is $2$-periodic and has Fourier expansion starting at $r^{-l}$. A simple computation shows that condition (d) implies (c) under the $2$-periodicity of $f$. Using the modular properties of $\Delta(z)$, we see that condition (d) in Proposition \ref{prop:Fconstruct} is equivalent to
$$
\p = \p_S+\p_T,
$$
which holds true.
\end{proof}

\begin{obs}
It is worth pointing out that by the fact that (see \cite[Remark after Theorem 8.4]{EZ85} and \cite[Equation (2.13)]{CKMRV21}) 
$$
M_k(\Gamma(2)) = M_k\oplus UM_{k-2}\oplus VM_{k-2}\oplus U^2M_{k-4}\oplus V^2M_{k-4}\oplus UVWM_{k-6},
$$
routine computations show that the kernel of $\p \mapsto\p_T+\p_S - \p$ in $M_k(\Gamma(2))$ coincides with
\begin{align*}
(U-V)M_{k-2} \oplus (U^2 - V^2)M_{k-4},
\end{align*}
which has dimension $ \lceil k/6 \rceil$ for $k\geq 6$, but is trivial otherwise. One can also show that 
$$
\dim \F_{d,l}=\bigg\lceil \frac{k}{6} \bigg\rceil-\frac{l}{2}=\frac{l}{2} -\bigg \lfloor \frac{d-4}{12} \bigg\rfloor.
$$
\end{obs}

\begin{lemma}[$\Gamma(1)$-construction]\label{lem:Ga1construction}
Let $d\geq 4$ be an integer divisible by $4$. Let $l\geq d/12$ be an even integer and set 
$$
k=2-d/2+6l  \quad \text{ and }  \quad  \ep=(-1)^{d/4}.
$$
Consider the following linear operator:
\begin{align*}
L: \psi \in M_{k+2}^{\leq 2} & \mapsto ([q^0] \psi,\, [q^{1}]\psi, ..., \, [q^{l/2}]\psi).
\end{align*}
Then any function in the vector space 
$$
\G_{d,l}:=\frac1{\Delta^{l/2}}(z^2\ker(L)_S),
$$
where $z^2\ker(L)_S:=\{z\mapsto z^2 \psi|_{k+2}S(z): \psi \in \ker(L)\}$, satisfies all conditions in Proposition \ref{prop:Fconstruct} for $\ep$, $d$ and $p(s)=\sin^2(\tfrac{\pi}{2} s)$. In particular, for any $g\in \G_{d,l}$, we have an associated radial Schwartz function $G_g:\r^d\to\cp$ such that $\ft {G_g}=\ep G_g$ and
$$
G_g(x) =\sin^2(\tfrac{\pi}{2} |x|^2)\int_0^{\infty i} g(z)e^{\pi i z |x|^2}\frac{\d z}{i}
$$
if $|x|^2>l$.
\end{lemma}
\begin{proof}
We will show that all conditions in Proposition \ref{prop:Fconstruct} are satisfied for $p(s)=\sin^2(\tfrac{\pi}{2} s)$. Conditions (a) and (b) follow in a similar manner as in the proof of Lemma \ref{lem:Ga2construction}. Setting $g=\psi|_kS/\Delta^{l/2}$, condition $(c)$ for $g$ is then equivalent to
$$
\psi|_kST^{-1}=\ep(-1)^{d/4}\psi|_kSTS,
$$
which is true since $\psi$ is $1$-periodic, $\ep(-1)^{d/4}=1$ and  $ST^{-1}(STS)^{-1}=T$.
Condition $(d)$ for $g$ is equivalent to
$$
\wt \psi(z)-\tfrac12(\wt \psi(z+1)+\wt \psi(z-1)) = \psi(z),
$$
where $\wt \psi(z) =\psi|_kS(z)$. 
However, simple computations show that this equation holds true for any function $\psi\in M_{k+2}^{\leq 2}$ using the characterization \eqref{eq:quasimodularcharac} and the functional equation of $E_2(z)$.
\end{proof}

\begin{obs}\label{obs:dimG}
It is not hard to show any function $g\in \G_{d,l}$ has an $q$-expansion of the form
$$
g=\sum_{n= -l/2}^\infty a_n q^{n},
$$
where $a_n$ is a quadratic polynomial in $z$ for $n\geq 1$, but affine for $n=-l,...,0$.
\end{obs}

The following lemma is reminiscent of the numerical method employed in \cite{CKMRV17}.

\begin{lemma}[Effective tail bounds]\label{lem:efflobounds}
Let $P(X,Y,Z)$ and $Q(X,Y,Z)$ be homogeneous polynomials of degree $k/2$ and $k+2$, respectively, where $k$ is even. Let $|P|$ and $|Q|$ denote the homogeneous polynomials derived from $P$ and $Q$, where each coefficient is replaced by its absolute value. Assume that $Q$ has no power of $X$ larger than $2$. Define the following holomorphic modular forms:
$$
\p= P(U,V,W)  \ \  \text{ and } \ \  \psi=  Q(E_2,E_4,E_6) .
$$
Let $\p_M$ and $\psi_M$ denote the tail of their respective $r$-series and $q$-series from the $(M+1)$-term onward.
Let 
$$
\p_S= P(U_S,V_S,W_S)
$$
and $(\p_S)_M$  be the tail of the $r$-series of the above function from the $r^{M+1}$-term Furthermore, letting $w=-\pi i z$ and
$$
w^2 \psi_S =  w^2 Q((E_2)_S,E_4,E_6),
$$
denote by $(w^2\psi_S)_M$ the tail of the $q$-series above from the $q^{M+1}$-term. Finally, let 
$$
R_M(p,j)=\sum_{n>M}{(n+1)^p}e^{-\pi{(n-j)}} \ \ \text{ and } \ \ S_M(p,j)=\sum_{n>M}{(n+1)^p}e^{-2\pi{(n-j)}}.
$$
Then for $j\leq M+1$ , $t\geq 1$, $r=e^{-\pi t}$ and $q=r^2$, we have:
\begin{enumerate}
\item  $|\p_M(it)| \leq |P|(8,8,8) R_M(\tfrac{3k-2}{2},j)r^j$;

\item $|\psi_M(it)| \leq |Q|(24,240,504) S_M(\tfrac{5k+10}{4},j)q^j$;

\item $|(\p_S)_{M}(it)| \leq |P|(8,8,8)R_{M}(\tfrac{3k-2}{2},j)r^{j}$;

\item $|(w^2 \psi_S)_M(it)| \leq 13|Q|(24,240,504)|S_M(\tfrac{5k+10}{4},j)q^j$.
\end{enumerate}
\end{lemma}
\begin{proof}
First we prove (1). Observe that by Jacobi's four-square theorem, the coefficient of $r^n$ in the $r$-series of each of the functions $U,V$ and $W$ is bounded by $8(n+1)^2$. Also note that whenever we multiply $m$ power series $\sum_{n\geq 0} (n+1)^{a_j} r^n$ for $j=1,...,m$, the coefficient of $r^n$ in the product is bounded by $(n+1)^{a_1+...+a_m+m-1}$. We deduce that the coefficient of $r^n$ in the $r$-series of $\p$ is bounded by $|P|(8,8,8) (n+1)^{(3k-2)/2}$. Since $t\geq 1$, we have $0\leq r\leq e^{-\pi}$. This easily implies item (1). The same argument shows item (3). Essentially the same argument shows item (2) by realizing that the coefficient of $q^n$ in the $q$-series of each of the functions $E_2,E_4$ and $E_6$ are bounded by $24(n+1)^2$, $240(n+1)^4$ and $504(n+1)^6$, respectively. We deduce that the coefficient of $q^n$ in the $q$-series of $\psi$ is bounded by $ |Q|(24,240,504) (n+1)^{a+b+c+k+1}$, where $a+b+c$ is maximal among nonnegative integers $a,b$ and $c$ such that $2a+4b+6c=k+2$ and $0\leq a\leq 2$. Greedy choice shows that the sum $a+b+c$ is maximized (or bounded by) when $(a,b,c)=(2,(k-2)/4,0)$, and thus, the coefficient of $q^n$ is bounded by $|Q|(24,240,504) (n+1)^{(5k+10)/4}$. This proves item (2).  Finally, item (4) is more nuanced since we have the presence of $w$. Let
$$
Q(X,Y,Z) = Q_0(Y,Z) + XQ_1(Y,Z)+X^2 Q_2(Y,Z),
$$
so that $|Q(24,240,504)|=|Q_0|(240,504)+24|Q_1|(240,504)+24^2|Q_2|(240,504)$ and
\begin{align*}
w^2 \psi_S & = w^2Q_0(E_4,E_6) + ({w^2}E_2-{6w})Q_1(E_4,E_6) + ({w}E_2-6)^2Q_2(E_4,E_6) \\
& = Qw^2 + (-12Q_2E_2 - 6Q_1)w + 36Q_2.
\end{align*}
For $z=it$ and $t\geq 1$, we have $\pi\leq w=\pi t \leq \tfrac1{7r}$, and we obtain
\begin{align*}
& {|(w^2 \psi_S)_M(it)|}  \\ 
& \leq (\tfrac1{7r})^2{|Q|(24,240,504) }S_M(\tfrac{5k+10}{4},j+1)q^{j+1} \\ & \quad +  \tfrac1{7r} (288|Q_2|(240,504)+6|Q_1|(240,504))S_M(\tfrac{5k}{4},j+1/2)q^{j+1/2}\\   & \quad +36|Q_2|(240,504)S_M(\tfrac{5k-10}{4},j)q^j \\
& \leq (\tfrac{e^{2\pi}}{49}{|Q|(24,240,504) }  +  \tfrac{e^{\pi}}{7} (288|Q_2|(240,504)+6|Q_1|(240,504))\\   & \quad +36|Q_2|(240,504))S_M(\tfrac{5k+10}{4},j)q^j 
\\ & <  (11|Q_0|(240,504) + 283|Q_1|(240,504) + 7283|Q_2|(240,504))S_M(\tfrac{5k+10}{4},j)q^j\\
& <13|Q|(24,240,504)S_M(\tfrac{5k+10}{4},j)q^j
\end{align*}
This proves item (4).
\end{proof}

\section{Proof of Theorem \ref{thm:extremallattices1} }
We assume Theorem \ref{thm:theguy}. Let $d=48$ and $K=\frac{1}{\sqrt{6}}([\sqrt{6},\sqrt{8}]\cup\{\sqrt{10}\})$. Let $\Lambda\subset \r^{48}$ be an even unimodular extremal lattice. In particular, $\Lambda$ is self-dual and $\ell(\Lambda)=\{\sqrt{6},\sqrt{8},...\}$. It is trivial to see $\Lambda$ is $K$-admissible and that
$$
\dens(\Lambda)= \vol \left( B_{48}\right) \left( \frac{\sqrt{6}}{2}\right)^{48}=\frac{(3\pi/2)^{24}}{24!}
$$
Consider $F(x)=H(\sqrt{6}x)$, with $H$ as in Theorem \ref{thm:theguy} for $d=48$. Poisson summation over $\Lambda$ shows that $H(0)=\ft H(0)>0$. The properties of $H$ imply that $F$ satisfies all conditions of Theorem \ref{thm:LP-bounds}, and hence, 
$$
\dens(\Lambda) \leq \Delta_d(K) \leq \vol(\tfrac12B_{48} ) \frac{F(0)}{\ft F(0)}= \vol \left( B_{48}\right) \left( \frac{\sqrt{6}}{2}\right)^{48}.
$$
This shows that equality above is attained and $\Lambda$ is optimal; that is, $\Delta_d(K) = \dens(\Lambda)$.

Now we prove uniqueness among all periodic packings. We follow the same strategy as in \cite[Section 8]{CE03}. Let $P=L+Y+\frac{1}{2}B_d$ be an optimal admissible periodic packing for some lattice $L$ and a set of translations $Y=\{v_1,...,v_M\}$. By Poisson Summation,

\begin{equation}
\sum_{j,l=1}^{M} \sum_{x\in L} F(x+v_j-v_l)=\frac{1}{|L|}\sum_{y\in L^*} \ft{F}(y)\left|\sum_{j=1}^{M} e^{2\pi i yv_j}\right|^2,
\end{equation}
from which we derive

\begin{equation}
M F(0)\geq\frac{1}{|L|}\ft{F}(0) M^2.
\end{equation}
Since $P$ is optimal, we have equality above, from which we derive that $|\sqrt{6}L|=M$ and that $\{x+v_j-v_l : x\in L, \, 1\leq j,l\leq M\}$ is contained in the set of zeros of $F$. By Theorem \ref{thm:theguy}, we deduce that   $\{\sqrt{6}|x+v_j-v_l| : x\in L, \, 1\leq j,l\leq M\}\subset\{0,\sqrt{6},\sqrt{8},\sqrt{10},...\}$. By \cite[Lemma 8.2]{CE03}, we conclude that the subgroup $G$ of $\r^{48}$ generated by the set $\sqrt{6}(L+Y)$ is an even integral lattice with minimal norm $\geq \sqrt{6}$. It also follows that the volume $|G|=\sqrt{N}$, for some integer $N$, and hence, $G$ has at most one point per unit of volume in $\r^{48}$. However, since $|\sqrt{6}L|=M$, the packing $\sqrt{6}(L+Y)+\frac{\sqrt{6}}{2}B_d$ has one sphere per unit of volume. Therefore, $G$ has exactly one point per unit volume, which implies that $N=1$, $G$ is unimodular and $G=\sqrt{6}(L+Y)$. Therefore, $G$ must be an extremal lattice.
This finishes the proof of Theorem \ref{thm:extremallattices1}.\hfill \qed

\section{Proof of Theorem \ref{thm:theguy}}

Let $8\leq d \leq 1200$ be a multiple of $8$, $a=a_d$ , $l= l_d$, $c=c_d$ and
$
K=K_d=\frac{1}{\sqrt{a}}\{\sqrt{a},\sqrt{a+2},...,\sqrt{l}\}
$.  In what follows, we set $k=2-d/2+6l$ (which is congruent to $2$ modulo $4$), $\ep=-1$,
$$
\b=\frac{l}{2} - \bigg\lfloor \frac{d-4}{12} \bigg\rfloor,
$$
$w=-\pi i z$, $r=e^{\pi i z}$ and $q=r^2$. We will abuse notation and write $u(z)=O(r^k)$ if for some $C>0$, we have $|u(z)|\leq C |z|^C e^{-\pi k \Im z }$ for $\Im z>C$.

The assertions below were done with rational arithmetic via PARI/GP \cite{gp} computer algebra system. Below, we will make certain claims, and we will indicate precisely how to prove them. This proof is computer assisted; hence, the necessary ancillary files can be found with the arXiv submission of this paper (arXiv:2308.03925). 

\noindent {\bf Step 1.} We apply Lemma \ref{lem:Ga2construction} and compute a basis for the vector space $\F_{d,l}$ collected as a row vector of functions
$$
[f_1,...,f_\b] = \Delta^{-l/2} [\p_1,\p_2,...,\p_\b],
$$
where $\dim \F_{d,l} = \b$ and
\begin{align*}
\Phi &= [\p_1,\p_2,...,\p_\b] = W^{l+1}(\M_\p {\bf \Theta}^\top)^\top, \\
\M_\p & \in \q^{\b \times (k/2-l )}, \\
{\bf \Theta} & = [W^{j-1} V^{k/2-l-j}]_{j=1,...,k/2-l}.
\end{align*}
We use the symbol $^\top$ for transpose. $\Theta$ is a row vector of basis functions for  $M_{k-2l-2}(\Gamma(2))$. It is easy to show that any function in $\Delta^{l/2}\F_{d,l}$ must be divisible by $W^{l+1}$, and this is the reason why we have isolated it in $\Phi$. To make sure $\Phi$ is uniquely defined (and so $\M_\p$) we normalize $\p_j$ so that 
$$
{\p_j} = r^{2(j-1)} + O(r^{2\b}).
$$

\noindent {\bf Step 2.} We apply Lemma \ref{lem:Ga1construction} and then proceed to find a basis
$$
[g_1,...,g_\cc] = \Delta^{-l/2} [w^2(\psi_1)_S,...,w^2(\psi_\cc)_S],
$$
collected as a row vector, for the subspace of functions $g\in \G_{d,l}$, with $g=w^2 \psi_S/\Delta^{l/2}$, and such that
$$
[w r^{j}](w^2 \Delta^{-l/2} \psi_S) =0
$$
for $j\in \{-l,...,-a\}$.  By the choice of $l$ and $a$, it turns out that $\cc=\b$ for all cases, we have computed (the proof that these dimensions coincide for all $d$ is lengthy and not worth to include here since we are doing this numerically anyways). Here, we set
\begin{align*}
\Psi & =  [\psi_1,\psi_2,...,\psi_\b] = (\M_\psi {\bf E}^\top)^\top, \\
\M_\psi & \in \q^{ \b \times (k+6)/4},\\
{\bf E} & = [E_2^{i}E_4^{j}E_6^n]_{\substack{2i+4j+6n=k+2 \\ j =0,...,(k+2)/4}}.
\end{align*}
$\bf E$ is a row vector of size $(k+6)/4$  that contains the basis functions for $M_{k+2}$ (for each $j=0,..,(k+2)/4$, the tuple $(i,n)$ is given by $i=((k+2)/2-2j) \text{ mod } 3$ and $n=(k+2-2i-4j)/6$). To make sure that $\Psi$ is uniquely defined, we impose that 
$$
{w^2(\psi_j)_S} = r^{2(j-1)} + O(r^{2\b})
$$
for $j=1,...,\b$. 

\noindent {\bf Step 3.} We now solve a linear system of homogeneous equations and set 
\begin{align*}
\p = \V_\p \M_\p {\bo \Theta}^\top\  \text{ and  } \ \psi & = \V_\psi \M_\psi {\bf E}^\top,
\end{align*}
where $ \V_\p$ and $\V_\psi$ are row vectors of size $\b$ (the solutions) that enforce the following $r$-expansion shapes:
\begin{align*}
{-w^2\psi _S- \p} &= \sum_{n= 0}^{l-a} \al_n r^{n} +\sum_{n=l - a+1}^{l} (\al_n + \al_n'w) r^{n}  + O(r^{l+1}),\\
 {-w^2\psi_S + \p} &= \sum_{n= l - a+1}^{l}  (\be_n + \be'_{n}w) r^{n} + O(r^{l+1}),
\end{align*}
for some $\al_n,\al_n',\be_n,\be_n'\in \q$. Recall that by Remark \ref{obs:dimG}, $[w^2 r^j](-w^2\psi_S)=0$ for $j\leq 0$. Note also that, by construction, we already have that $\al_n,\al_n',\be_n,\be_n'$ vanish for odd $n$ in the range $0\leq n\leq l$.
More precisely, given all the previous constraints,  $ \V_\p$ and $\V_\psi$ are solutions of the homogeneous equations
$$
[w^0 r^j]({-w^2\psi_S + \p})=0
$$ 
for $j\in \{0,2,...,l-a\}$. It turns out that $\V_\p$ and $\V_\psi $ are uniquely defined modulo scaling. We then define the row vectors
$$
C_\p = {n}\V_\p \M_\p \ \text{ and } \ C_\psi={n}\V_\psi \M_\psi,
$$
where we choose $n\in \z$ so that $C_\p$ and $C_\psi $ are vectors of integers where $\gcd(C_\p\cup C_\psi)=1$ and the first nonzero coordinate of $C_\p$ is positive. In this way, $C_\p$ and $C_\psi $ are uniquely defined. For instance, for $d=48$, we have
\begin{align*}
C_\p &=2^7\times 3^8 \times [29393, 117572, 307819, 511955, 539410, 362729, 152114, 36480, 3840] \\
C_\psi & = [565675, 7394933, -38880096, 44550063, 41316945, -107522880, 39169185, \\ & \quad \ \ 40077567, -32756064, 5294597, 790075].
\end{align*}
These vectors cannot be simplified much further nor have some easy to guess combinatorial formula since, for instance, the $8$th entry of $C_\psi$ is divisible by the large prime $4453063$ and $7$th entry of $C_\p$ is divisible by the prime $4003$. Experimentally, large primes are often found in the vectors $C_\p$ and $C_\psi$ as dimension grows. A list of all vectors  $C_\p$ and $C_\psi$ for each dimension $d\leq 1200$ multiple of $8$ can be found on the ancillary files in the arXiv submission of this paper (a file named \emph{Cvectors}).

\noindent {\bf Step 4.} We  use Lemmas \ref{lem:Ga2construction} and \ref{lem:Ga1construction} to create a radial Schwartz function $H:\r^d\to\r$ and obtain the integral representations
\begin{align*}
H(x)& =\sin^2(\tfrac{\pi}{2} |x|^2)\int_0^{\infty } \frac{-\pi^2 t^2\psi_S(it)-\p(it)}{\Delta^{l/2}(it)}e^{- \pi t |x|^2}\d t \\
\ft H(x)& =\sin^2(\tfrac{\pi}{2} |x|^2)\int_0^{\infty}  \frac{-\pi^2 t^2\psi_S(it)+\p(it)}{\Delta^{l/2}(it)}e^{ -\pi t |x|^2}\d t,
\end{align*}
that hold for $|x|^2>l$. This shows that
$$
H(x)=\ft H(x)=0 \text{ for } |x|^2 \in \{l+2,l+4,...\}.
$$
However, if we let $g=-\Delta^{-l/2} w^2 \psi_S$ and $f=\Delta^{-l/2} \p$, then, by construction, the $r^j$-coefficient of $g-f$ is a rational number for $j=-l,...,-a$.  Similarly, by construction, the $r^j$-coefficient of $g+f$ vanishes for $j=-l,...,-a$. A straightforward computation shows also that $H(x)=\ft H(x)=0$ for $|x|^2\in \{a,a+2,...\}$, and the integral representation of $\ft H(x)$ above converges for $|x|^2>a-2$.

\noindent {\bf Step 5.} 
From now on, we assume that $d \nequiv 16 \mod 24$. We claim that
\begin{itemize}
\item[(i)] $|\pi^2 \psi(it)|< -\p_S(it)$ for $t\geq 1$;
\item[(ii)] $|\pi^2 t^2\psi_S(it)| < \p(it)$  for $t\geq 1$.
\end{itemize}
Notice that for $|x|^2>a-2$, we have

\begin{align*}
\frac{\ft H(x)}{\sin^2(\tfrac{\pi}{2} |x|^2)} & = \int_1^{\infty } \frac{-\pi^2 t^2\psi_S(it)+\p(it)}{\Delta^{l/2}(it)}e^{- \pi t |x|^2}\d t +  \int_1^{\infty } \frac{-\pi^2\psi(it)-\p_S(it)}{\Delta^{l/2}(it)}e^{- \pi  |x|^2/t} \frac{\d t}{t^{d/2}} > 0,
\end{align*}
by  conditions (i) and (ii).  This implies that $\ft H(x) \geq 0$ for $|x|^2>a-2$ and vanishes exactly at $|x|^2 \in \{a,a+2,a+4,...\}$ if $|x|^2>a-2$. Similarly, for $|x|^2>l$, we have 
\begin{align*}
\frac{H(x)}{\sin^2(\tfrac{\pi}{2} |x|^2)} & = \int_1^{\infty } \frac{-\pi^2 t^2\psi_S(it)-\p(it)}{\Delta^{l/2}(it)}e^{- \pi t |x|^2}\d t +  \int_1^{\infty } \frac{-\pi^2\psi(it)+\p_S(it)}{\Delta^{l/2}(it)}e^{- \pi  |x|^2/t} \frac{\d t}{t^{d/2}}
 < 0,
\end{align*}
and so $H(x) \leq 0$ for $|x|^2>l$ and vanishes exactly at $|x|^2 \in \{l,l+2,l+4,...\}$ if $|x|^2>l-\ep$, for some small $\ep>0$.

\noindent {\bf Step 6.} To prove the claims (i) and (ii) in Step 5, we introduce the following notation: For a given $u=\sum_{n\geq 0} a_n(w) r^n$, we write
$$
\tu{u} = \sum_{n= 0}^N a_n(w) r^n + r^{l+10} \  \text{ and }  \ 
\td{u} = \sum_{n= 0}^N a_n(w) r^n - r^{l+10}.
$$
Recall from Lemma \ref{lem:efflobounds} that $R_N$ and $S_{N/2}$ are the corresponding tail sums. We will choose $N \geq l+10n$, with $n\geq 1$, to be the first integer such that the quantity
$$
\max\bigg\{{\rm abs}(C_\p){\bf \Theta}^\top|_{r=0}R_N((3k-2)/2,l+10), 13{\rm abs}(C_\psi ){\bf E}^\top|_{q=0}S_{N/2}((5k+10)/4,l/2+5)\bigg\}
$$
is less than $1$. Above, ${\rm abs}(v)$, for a vector $v$, is simply the same vector with each coordinate replaced by its absolute value. By Lemma \ref{lem:efflobounds}, this guarantees that 
\begin{align*}
\td\p &\leq \p \leq \tu \p,  \quad \quad \quad \, \td{\p_S} \leq \p_S \leq \tu {\p_S},\\
\td{\psi} &\leq \psi \leq \tu {\psi},  \quad \quad \td{w^2\psi_S} \leq w^2\psi_S \leq \tu {w^2\psi_S},
\end{align*}
for $z=it$, $w=\pi t$ and $t\geq 1$. A list of all $N$'s for each dimension $d$ can be found on the ancillary files in the arXiv version of this paper (a file named \emph{Nnumbers}). For instance, for $d=48$, we have $N=130$.

To prove that condition (i) is satisfied, first we verify that\footnote{We still do not fully know why this happens, but it is true for every case we have computed.}
\begin{enumerate}
\item[({I})] \ $ C_\p \geq 0$,
\end{enumerate}
since it directly shows that $\p(it)>0$ for $t>0$, which implies that (since $k/2$ is odd) $\p_S(it)<0$ for $t>0$. Next, we show that\footnote{Another mystery, but it is true for every case we have computed.}
\begin{enumerate}
\item[({II})] \   $\td \psi$ has only nonnegative coefficents in its $q$-expansion,
\end{enumerate}
which proves that $\psi(it)>0$ for $t\geq 1$ and that $-\pi^2 \psi(it)<-\p_S(it)$ for $t\geq 1$. Next, we use Sturm's method (which can be done via exact rational arithmetic evaluations)  on the variable $r$ to show that the polynomial
\begin{enumerate}
\item[({III})] \  $- \pi_2^2 \tu \psi - \tu{\p_S} > 0$ for $0 < r< \ga_2$,
\end{enumerate}
where we use $\pi_2 = \lceil \pi 10^{m}\rceil 10^{-m}$ for $m=20,40,...,100$ and $\ga_2=\lceil e^{-\pi} 10^{m'}\rceil 10^{-m'}$ for $m'=2,5,8,11$, where we select $(m,m')$ according to necessity (for large dimensions, more precision is sometimes required). From now on, whenever we apply Sturm's method in the range $0 < r< \ga_2$, we will select $\ga_2$ as before. This shows that  $\pi^2 \psi(it) <- {\p_S}(it)$ for $t\geq 1$, and proves that condition (i) holds. 

\noindent {\bf Step 7.} For condition (ii), we write
\begin{align*}
\td{w^2 \psi_S}+\td \p  & = [1,w,w^2][P_0(e^{-w}),P_1(e^{-w}),P_2(e^{-w})]^\top\\
-\tu{w^2 \psi_S}+\td \p & = [1,w,w^2][Q_0(e^{-w}),Q_1(e^{-w}),Q_2(e^{-w})]^\top,
\end{align*}
where the $P_i$'s  and $Q_i$'s are polynomials with integer coefficients and degree at most $N$. 
Note that $P_2(e^{-w})=\td \psi (it)$ and $Q_2(e^{-w})=-\td \psi (it)-e^{-w(l+10)}$, so $-Q_2$ and $P_2$ have only positive coefficients. We then set $x=e^{-w}$ and 
\begin{align*}
w_1(x)  =\sum_{n=1}^N \frac{(1-x)^n}{n} \text{ and } w_2(x)  = \sum_{n=1}^N \frac{(1-x)^n}{n}+\frac{(1-x)^{N+1}}{(N+1)x},
\end{align*}
and note that $w_1(x)<w < w_2(x)$ for $0<x<1$. We then conclude that condition (ii) is implied by the two conditions below:

\begin{enumerate}
\item[({IV})]  \  $[1,w_{j}(x),w_{1}(x)^2][P_0(x),P_1(x),P_2(x)]^\top>0$ for all $ j\in \{1,2\}$ and $0<x<\ga_2$;
\item[({V})]  \  $[1,w_{j}(x),w_{2}(x)^2][Q_0(x),Q_1(x),Q_2(x)]^\top>0$ for all $ j\in \{1,2\}$ and $0<x<\ga_2$.
\end{enumerate}
Both can now be verified using Sturm's method.

\noindent {\bf Step 8.}  The proof that $\ft H(x)>0$ for $c<|x|^2 < a-2$ is more involved. Let $v(w) = -{w^2 \psi_S}+\p$ and $\Delta^{-l/2}=\sum_{n\geq -l} \delta_{l,n} r^n$. For $s=|x|^2>a-2$, we have
\begin{align*}
 \frac{\ft H(x)}{\sin^2(\tfrac{\pi}{2} |x|^2)} & = \int_1^{\infty } \frac{v(\pi t)}{\Delta^{l/2}(it)}e^{- \pi t s}\d t +  \int_1^{\infty } \frac{-\pi^2\psi(it)-\p_S(it)}{\Delta^{l/2}(it)}e^{- \pi  |x|^2/t} t^{-d/2}\d t \\
& =  \int_1^{\infty } v(\pi t)\bigg(\sum_{n=-l}^{N} \delta_{l,n} e^{-n\pi t}\bigg)e^{- \pi t s}\d t + \int_1^{\infty } v(\pi t)\bigg(\sum_{n>N} \delta_{l,n} e^{-n\pi t}\bigg)e^{- \pi t s}\d t \\ & \quad +  \int_1^{\infty } \frac{-\pi^2\psi(it)-\p_S(it)}{\Delta^{l/2}(it)}e^{- \pi  s/t} t^{-d/2}\d t .
\end{align*}
In this way, the last two integrals above converge absolutely for $s=|x|^2>0$, while the first integral extends to a meromorphic function of $s\in \cp$ with possible poles $s=a-2,a-4,...,2,0$. Since $\ft H(x)$ is entire in the variable $s$, the above representation now holds in the region $\re s \in (0,\infty)\setminus \{2,4,...,,a-2\}$. In particular, since $\delta_{l,n}\geq 0$, $v(\pi t)>0$ and $-\pi^2\psi(it)-\p_S(it)>0$ for $t\geq 1$, we obtain the following inequality in the range $s>0$:
$$
\frac{\ft H(x)}{\sin^2(\tfrac{\pi}{2} |x|^2)} > \wt{ \int}_{1}^\infty {v(\pi t)}\bigg(\sum_{n= -l}^{N-l-2} \delta_{l,n} e^{-n\pi t} \bigg) e^{-\pi t s} \d t ,
$$
where by $ \wt{ \int}$ we mean the meromorphic extension of the function defined by this integral.  Now let $A(w)$ be such that $v(w)-A(w)=O(r^l)$. Then the right-hand side above is

\begin{align*}
& \wt{ \int}_{1}^\infty A(\pi t)\bigg(\sum_{n= -l}^{N-l-2} \delta_{l,n} e^{-n\pi t} \bigg) e^{-\pi t s} \d t  + { \int}_{1}^\infty (v(\pi t)-A(\pi t))\bigg(\sum_{n= -l}^{N-l-2} \delta_{l,n} e^{-n\pi t} \bigg) e^{-\pi t s} \d t 
 \\ & >  \wt{ \int}_{1}^\infty A(\pi t)\bigg(\sum_{n= -l}^{N-l-2} \delta_{l,n} e^{-n\pi t} \bigg) e^{-\pi t s} \d t  \\ & \quad \quad + { \int}_{1}^\infty (-\tu{w^2 \psi_S}(it)+\td{\p}(it)-A(\pi t ))\bigg(\sum_{n= -l}^{N-l-2} \delta_{l,n} e^{-n\pi t} \bigg) e^{-\pi t s} \d t  \\
 & = \wt{ \int}_{1}^\infty (-\tu{w^2 \psi_S}(it)+\td{\p}(it))\bigg(\sum_{n= -l}^{N-l-2} \delta_{l,n} e^{-n\pi t} \bigg) e^{-\pi t s} \d t,
\end{align*}
for $s>0$. After a change of variables $w=\pi t$, we conclude that
\begin{align*}
 \frac{\pi \ft H(x)}{\sin^2(\tfrac{\pi}{2} |x|^2)} > \wt{\int}_{\pi}^\infty \bigg( \sum_{m=-a+2}^{2N-l-2} p_m(w)e^{-mw} \bigg) e^{-w s}\d w
\end{align*}
for $s>0$, where
$$
(-\tu{w^2 \psi_S}(it)+\td{\p}(it))\bigg(\sum_{n= -l}^{N-l-2} \delta_{l,n} e^{-n\pi t} \bigg) =: \sum_{m=-a+2}^{2N-l-2} p_m(w)e^{-mw}.
$$
Let now $M=\lfloor (2N-l-2)/2^{m}\rfloor$ with $m=4,3,2,1,0$ (depending on precision). Let
\begin{align*}
A_1(w)&=\sum_{m=-a+2}^{M-1} p_m(w)e^{-mw} + e^{-Mw}w^2[w^2](p_M(w)) \\
A_2(w)&=-A_1(w)+\sum_{m=-a+2}^{2N-l-2} p_m(w)e^{-mw}.
\end{align*}
Let $x=e^{-w}$. We then show that \begin{enumerate}
\item[({VI})]  \  $[1,w_{j}(x),w_{2}(x)^2][[w^0]A_2,[w^1]A_2,[w^2]A_2]^\top>0$ for all $ j\in \{1,2\}$ and $0<x<\ga_2$.
\end{enumerate}
This proves that
\begin{align}\label{ineqHhat}
 \frac{\pi \ft H(x)}{\sin^2(\tfrac{\pi}{2} |x|^2)} > \wt{\int}_{\pi}^\infty A_1(w) e^{-w s}\d w.
\end{align}
Let now $\R$ be the following `rationalization' operator
$$
P(s)=\sum c_{i,j,n} \pi^i e^{-\pi j} s^n \mapsto \R(P)(s):= \sum \wt c_{i,j,n} s^n,
$$
where $\wt c_{i,j,n}=\min_{\al,\be \in \{1,2\}}\{c_{i,j,n} \pi_\al^i \ga_\be^j\}$ and $\pi_1,\pi_2,\ga_1,\ga_2$ are rational approximations of $\pi$ and $\ga=e^{-\pi}$ such that
$$
\pi_1 < \pi < \pi_2 \quad \text{and} \quad \ga_1 < \ga<\ga_2.
$$
Usually these rational approximations are taken to be $m$-digit truncations (in base $10$)  from below and above, with $m\in \{10,15,20,...,50\}$ depending on the required precision.  Observe now that if $p(w)$ is a quadratic polynomial with coefficients in $\q[\pi]$, then we obtain that
\begin{align*}
& e^{\pi s} \int_\pi^\infty p(w)e^{-mw}e^{-sw} \d w \\ & = e^{-\pi m}\int_0^\infty p(w+\pi)e^{-(s+m)w} \d w \\ & = e^{-\pi m} \sum_{j=0}^2 \frac{j![w^j](p(w+\pi))}{(s+m)^{j+1}} \\
& = e^{-\pi m} \frac{(s+m)^3[w^0](p(w+\pi))+(s+m)^2[w^1](p(w+\pi))+2(s+m)[w^2](p(w+\pi))}{(s+m)^4} \\
& \geq \frac{\R(e^{-\pi m}(s+m)[w^0](p(w+\pi)))}{(s+m)^2}+\frac{\R(e^{-\pi m}[w^1](p(w+\pi)))}{(s+m)^2}\\  &\quad  +\frac{\R(2e^{-\pi m}(s+m)[w^2](p(w+\pi)))}{(s+m)^4}\\
& =: {B_m[p](s)}.
\end{align*}
We deduce that \eqref{ineqHhat} is bounded from below by
\begin{align*}
 e^{-\pi s}\bigg(\sum_{m=-a+2}^{M-1} B_m(p_m)(s) + B_M(w^2[w^2](p_M(w))) \bigg)  =: e^{-\pi s}Q(s),
\end{align*}
where $Q(s)$ is a rational function with rational coefficients. Finally, we write $Q=Q_{num}/Q_{den}$, for polynomials $Q_{num}$ and $Q_{den}$ and obtain that $\ft H(x)>0$ for $c<|x|^2 < a-2$ holds true if the following condition is satisfied
\begin{enumerate}
\item[({VII})]  $Q_{den}(s) \prod_{j=0}^{a/2-1} (s-2j)^{-2}$ has only nonnegative coefficients and $Q_{num}(s)>0$ for $c<s<a-2$.
\end{enumerate}
This can be checked by Sturm's method again. Notice that since we have used only strict inequalities, this shows that $\{|x|^2 : \ft H(x)=0 \text{ and } |x|>c_d\}=\{a_d,a_d+2,...\}$.

We then check that conditions (I), (II), (III), (IV), (V), (VI) and (VII) are satisfied using rational arithmetic only, producing in this way a mathematical proof. The necessary algorithm to check this positivity conditions can be found in the ancillary files in the arXiv submission of this paper (a file named \emph{Postest}; please also read the file \emph{Readme}).

\noindent {\bf Step 9.} Finally, for $d=48$, it remains to show the claim that $\{|x|^2:H(x)< 0\}\cap (0,10) = (6,8)$. The method is exactly the same as the one employed on Step 8, except we start from $H(x)$ and $-{w^2 \psi_S}-\p$ and show, after essentially the same procedure, that the resulting rational function $Q(s)$ divided by $(s-6)(s-8)$ is positive in the interval $0<s<10$.

This finishes the proof of Theorem \ref{thm:theguy}.\hfill \qed

\section{Proof of Theorems \ref{thm:extremallattices2} and \ref{cor:theguy} }
Noting that $A_d = (1,\sup(K_d)]\setminus K_d$, we conclude that a packing $P$ is $K$-admissible if and only if it avoids $A_d$. Thus, Theorems \ref{thm:extremallattices2} and \ref{cor:theguy} are equivalent.
The proof of Theorem \ref{cor:theguy} follows directly from Theorem \ref{thm:theguy}. Let $P=\Lambda+Y+\bd$ be a $K_d$-admissible periodic sphere packing. Poisson summation over $\Lambda+Y-Y$ with the function $F(x)=H(\sqrt{a_d}x)$ and $H$ as in Theorem \ref{thm:theguy} shows that 
$$
H(0)\#Y \geq \sum_{y,y'\in Y} \sum_{x\in L} F(x+y-y') = \frac{1}{\vol(\r^d/L)} \sum_{x^*\in L^*} \ft F(x^*)\bigg|\sum_{y\in Y} e^{2\pi i y \cdot x^*}\bigg|^2 \geq  \frac{\ft H(0)a_d^{-d/2}(\#Y)^2}{\vol(\r^d/L)},
$$
hence,
$\dens(P)\leq \vol \left( B_d\right) \left( \frac{\sqrt{a_d}}{2}\right)^d$, which is attained by any even unimodular extremal lattice. Poisson summation implies that equality is attained for a lattice packing ($\#Y=1$) if and only if for every $v\in \Lambda^*$, we have $|v|^2\in \{a_d,a_d+2,...\}$. Then, as in the proof of Theorem \ref{thm:extremallattices1}, one shows that $\sqrt{a_d}\Lambda$ is an even unimodular extremal lattice.\hfill \qed



\section*{Acknowledgements}
The authors thank João P. Ramos, Henry Cohn and Danlyo Radchenko for fruitful discussions on the elaboration of this paper.  The first author acknowledges support from the following funding agencies: The Office of Naval Research GRANT14201749 (award number N629092412126), The Serrapilheira Institute (Serra-2211-41824), FAPERJ (E-26/200.209/2023) and CNPq (309910/2023-4). The second author acknowledges the support of CNPq (141446/2023-4) and FAPERJ (E-26/202.492/2022) scholarships.

{\bf Conflict of interest} None.

{\bf  Data availability statement}
Ancillary files with code and data are available in the arXiv submission of this paper: arXiv:2308.03925 [math.NT].


\end{document}